\documentclass[11pt]{amsart}
\usepackage{amssymb, adjustbox, enumerate, amsbsy, stmaryrd}
\usepackage{geometry}

\usepackage{todonotes}

\usepackage{amsfonts, amssymb, amscd}
\numberwithin{equation}{section}

\usepackage[symbol]{footmisc}

\usepackage{bm}
\usepackage{verbatim}
\usepackage{mathrsfs}
\usepackage{graphicx}
\usepackage{tikz-cd}
\usepackage{subcaption}
\usepackage{listings}
\usepackage{subfiles}
\usepackage[toc,page]{appendix}
\usepackage{mathtools}
\usepackage{comment}
\usepackage{enumerate}
\usepackage{enumitem}
\usepackage[all]{xy}

\usepackage{graphicx}
\graphicspath{{images/}}

\usepackage{appendix}
\usepackage{hyperref}
\hypersetup{
    colorlinks=true,
    citecolor=red,
    linkcolor=blue,
    filecolor=magenta,      
    urlcolor=red,
}
\newcommand{\Dd}{{\mathcal{D}}}

\newcommand{\id}{\mathrm{id}}

\newcommand{\Cc}{\mathbb{C}}

\newcommand{\Qq}{\mathbb{Q}}

\newcommand{\Rr}{\mathbb{R}}

\newcommand{\Zz}{\mathbb{Z}}

\newcommand{\vol}{\operatorname{vol}}
\newcommand{\Center}{\operatorname{center}}

\newcommand{\Exc}{\operatorname{Exc}}

\newcommand{\ct}{\operatorname{ct}}

\newcommand{\mld}{{\rm{mld}}}
\newcommand{\relin}{\operatorname{relin}}

\newcommand{\lct}{\operatorname{lct}}

\newcommand{\Supp}{\operatorname{Supp}}

\newcommand{\codim}{\operatorname{codim}}
\newcommand{\mult}{\operatorname{mult}}

\newcommand{\Ii}{\Gamma}

\newcommand{\mm}{\mathfrak{m}}

\newcommand{\CT}{\operatorname{CT}}

\newcommand{\coeff}{\operatorname{coeff}}

\newtheorem{thm}{Theorem}[section]
\newtheorem{conj}[thm]{Conjecture}
\newtheorem{cor}[thm]{Corollary}
\newtheorem{lem}[thm]{Lemma}
\newtheorem{sett}[thm]{Setting}

\theoremstyle{definition}
\newtheorem{defn}[thm]{Definition}

\theoremstyle{definition}
\newtheorem{rem}[thm]{Remark}

\theoremstyle{definition}

\begin{document}

\title[Termination of flips and enc pairs]{On termination of flips and exceptionally non-canonical singularities}
\author{Jingjun Han and Jihao Liu}

\date{\today}

\begin{abstract}
We systematically introduce and study a new type of singularities, namely, exceptionally non-canonical (enc) singularities. This class of singularities plays an important role in the study of many questions in birational geometry, and has tight connections with local K-stability theory, Calabi-Yau varieties, and mirror symmetry.

We reduce the termination of flips to the termination of terminal flips and the ACC conjecture for minimal log discrepancies (mlds) of enc pairs. As a consequence, the ACC conjecture for mlds of enc pairs implies the termination of flips in dimension $4$.

We show that, in any fixed dimension, the termination of flips follows from the lower-semicontinuity for mlds of terminal pairs, and the ACC for mlds of terminal and enc pairs. Moreover, in dimension $3$, we give a rough classification of enc singularities, and prove the ACC for mlds of enc pairs. These two results provide a second proof of the termination of flips in dimension $3$ which does not rely on any difficulty function.

Finally, we propose and prove the special cases of several conjectures on enc singularities and local K-stability theory. We also discuss the relationship between enc singularities, exceptional Fano varieties, and Calabi-Yau varieties with small mlds or large indices via mirror symmetry.
\end{abstract}

\address{Shanghai Center for Mathematical Sciences, Fudan University, Jiangwan Campus, Shanghai, 200438, China}
\email{hanjingjun@fudan.edu.cn}

\address{Department of Mathematics, Peking University, No. 5 Yiheyuan Road, Haidian District, Peking 100871, China}
\email{liujihao@math.pku.edu.cn}

\maketitle

\tableofcontents

\section{Introduction}

We work over the field of complex numbers $\mathbb C$.

The minimal model program (MMP) aims to provide a birational classification of algebraic varieties. The termination of flips in the MMP is one of the major remaining open problems in birational geometry. The goal of this paper is to introduce and study a class of singularities, namely, \emph{exceptionally non-canonical (enc)} singularities, and utilize it on the termination of flips and other topics in birational geometry. A pair $(X,B)$ is called \emph{enc} if $(X,B)$ is not canonical, and all but one exceptional prime divisors over $X$ have log discrepancies greater than or equal to $1$ (see Definition \ref{defn: enc}).

\medskip

\noindent\textbf{Termination of flips and the ACC for mlds of enc pairs}. Shokurov established a relation between the termination of flips and minimal log discrepancies (mlds), a basic but important local invariant in birational geometry. To be specific, Shokurov \cite[Theorem]{Sho04} proved that his ACC conjecture for (local) mlds \cite[Problem 5]{Sho88} together with the lower-semicontinuity (LSC) conjecture for mlds \cite[Conjecture 0.2]{Amb99} imply the termination of flips.

On the one hand, the ACC conjecture for mlds remains unknown even in dimension $3$, while the termination of flips is proved for threefolds \cite{Kaw92a,Sho96}. On the other hand, the ACC conjecture for mlds aims to reveal the structure of mlds of \emph{all} singularities, while the singularities appearing in any given sequence of flips should be very special (even of finitely many types as we conjecture that the sequence of flips terminates). This indicates that the ACC conjecture for mlds might be much more difficult than the termination of flips, and we may not need the full power of this conjecture to show the termination of flips.

In this paper, we try to resolve this issue. Our first main result reduces the termination of flips to the ACC for (global) mlds of a very special class of singularities, enc singularities, and the termination of terminal flips. The latter is known up to dimension $4$.

\begin{thm}\label{thm: enc and tof}
Let $d$ be a positive integer. Assume that
\begin{enumerate}
    \item the ACC for (global) mlds of enc pairs with finite coefficients of dimension $d$ holds, i.e. $$\{\mld(X,B)\mid (X,B)\text{ is enc, }\dim X=d, \coeff(B)\subseteq\Ii\}$$
    satisfies the ACC where $\Ii\subset [0,1]$ is a finite set (Conjecture \ref{conj: acc mld enc}(2')), and
    \item any sequence of $\mathbb Q$-factorial terminal flips in dimension $d$ terminates.
\end{enumerate}
Then any sequence of lc flips in dimension $\le d$ terminates.
\end{thm}

We note that our proof of Theorem \ref{thm: enc and tof} does not rely on \cite{Sho04}. As a consequence of Theorem \ref{thm: enc and tof}, the ACC for (global) mlds of enc pairs implies the termination of flips in dimension $4$. 

\begin{thm}\label{thm: enc and tof dim 4}
Assume the ACC for (global) mlds of enc pairs with finite coefficients of dimension $4$. Then any sequence of lc flips in dimension $4$ terminates.
\end{thm}

As another application of Theorem \ref{thm: enc and tof}, we may refine Shokurov's approach towards the termination of flips. To be more specific, assuming the ACC for (global) mlds of enc pairs, in order to prove the termination of flips, we only need the ACC and LSC for mlds of terminal pairs instead of lc pairs. We note that the set of terminal singularities is the smallest class for the purpose to run the MMP for smooth varieties, while the set of lc singularities is rather complicated and hard to work with (see \cite[page 57]{KM98}). Moreover, terminal surface (resp. threefold) singularities are smooth (resp. hyperquotient singularities), and the LSC conjecture for mlds is proven for the smooth varieties and hyperquotient singularities in any dimension \cite{EMY03,NS21}.

\begin{thm}[$=$Theorem \ref{thm: ter acc lsc+enc to tof}]\label{thm: strong lsc+acc to tof}
Let $d$ be a positive integer. Assume
\begin{enumerate}
    \item the ACC for (local) mlds of terminal pairs with finite coefficients,
    \item the LSC for mlds of terminal pairs, and
    \item the ACC for (global) mlds of enc pairs with finite coefficients (Conjecture \ref{conj: acc mld enc}(2'))
\end{enumerate}
hold in dimension $d$. Then any sequence of lc flips in dimension $\le d$ terminates.
\end{thm}
The set of enc singularities is expected to be a really small class of singularities that should possess some nice properties. The local Cartier indices of enc singularities are expected to be bounded from above (see Conjecture \ref{conj: index of enc}), hence we predict that the set of their mlds is discrete away from zero and should be much smaller than the set of all mlds. We also remark that in Theorems \ref{thm: enc and tof}, \ref{thm: enc and tof dim 4}, and \ref{thm: strong lsc+acc to tof}, we only need the ACC for (global) mlds (of enc pairs), which is considered to be much simpler than the ACC for (local) mlds. For instance, consider any normal variety $X$. The log discrepancy of the exceptional divisor, obtained through the blow-up of $X$ at any smooth codimension $2$ point, is $2$. Therefore, the global mld of any normal variety is always $\le 2$. However, the boundedness conjecture for (local) mlds remains open in dimension $\ge 4$.

\medskip

\noindent\textbf{ACC for mlds of enc threefolds}. Recall that the ACC conjecture for mlds is only known in full generality for surfaces \cite{Ale93} (see \cite{Sho94b,HL23b,HL24} for other proofs), toric pairs \cite{Amb06}, and exceptional singularities \cite{HLS24}. It is still open for threefolds in general and only some partial results are known (see \cite{Kaw92b,Mar96,Kaw15b,Nak16,NS22,Jia21,LX21,HLL22,LL22}). The second main result of this paper is the ACC for mlds of enc pairs in dimension $3$. This result suggests that the ACC conjecture for mlds of enc pairs should be much easier than Shokurov's ACC conjecture for mlds.

\begin{thm}[cf. Theorem \hyperlink{thm: E}{E}]\label{thm: acc mld enc dim 3}
Let $\Ii\subset [0,1]$ be a DCC set. Then
$$\{\mld(X,B)\mid (X,B)\text{ is enc, }\dim X=3, \coeff(B)\subseteq\Ii\}$$
satisfies the ACC.
\end{thm}

By Theorems \ref{thm: strong lsc+acc to tof} and \ref{thm: acc mld enc dim 3}, we may reprove the termination of flips in dimension $3$.

\begin{cor}[$=$Corollary \ref{cor: new proof tof dim 3}]\label{cor: tof dim 3}
Any sequence of lc flips terminates in dimension $3$.
\end{cor}

We remark that our proof of Corollary \ref{cor: tof dim 3} only depends on the ACC and the LSC for mlds, and does not rely on any other auxiliary methods. In particular, our proof does not rely on difficulty functions, which played a key role in the previous proofs on the termination of flips in dimension $3$ (see \cite{Sho85,KMM87,Kaw92a,Kol+92,Sho96}) but are difficult to be applied in higher dimensions, especially in dimension $>4$ (we refer the reader to \cite{AHK07,Sho04,Fuj05} for some progress in dimension $4$). The proof of Corollary \ref{cor: tof dim 3}, which is based on Theorem \ref{thm: strong lsc+acc to tof}, may shed light on another approach towards the termination of flips in high dimensions. Note that Corollary \ref{cor: tof dim 3} does not follow from Shokurov's approach \cite{Sho04} directly, as the ACC conjecture for mlds is still open in dimension $3$.

\medskip

Now we turn our attention away from the termination of flips and focus on the ACC conjecture for mlds itself. We may show the following more technical but much stronger result on the ACC conjecture for mlds.

\begin{thm}[cf. Theorem \hyperlink{thm: N}{N}]\label{thm: acc mld 3dim bounded ld number}
Let $N$ be a non-negative integer, and $\Ii\subset [0,1]$ a DCC set. Then there exists an ACC set $\Ii'$ depending only on $N$ and $\Ii$ satisfying the following. Assume that  $(X,B)$ is a klt pair of dimension $3$ such that
\begin{enumerate}
    \item $\coeff(B)\subseteq\Ii$, and
    \item there are at most $N$ different (exceptional) log discrepancies of $(X,B)$ that are $\le 1$, i.e.
    $$\#(\{a(E,X,B)\mid E\text{ is exceptional over }X\}\cap [0,1])\le N,$$
\end{enumerate}
then $\mld(X,B)\in\Ii'$.
\end{thm}

Theorem \ref{thm: acc mld 3dim bounded ld number} is considered to be much stronger than Theorem \ref{thm: acc mld enc dim 3} as a result on the ACC conjecture for mlds, and the class of singularities in Theorem \ref{thm: acc mld 3dim bounded ld number} is much larger than the class of enc singularities. It is clear that the local Cartier indices of these singularities in Theorem \ref{thm: acc mld 3dim bounded ld number} are unbounded when their mlds have a positive lower bound, while they are expected to be bounded for enc pairs (cf. Conjecture \ref{conj: index of enc}). Moreover, when $N=0$, Theorem \ref{thm: acc mld 3dim bounded ld number} implies the ACC for mlds of terminal threefolds \cite[Theorem 1.1]{HLL22} which is beyond Theorem \ref{thm: acc mld enc dim 3}, and when $N=1$, Theorem \ref{thm: acc mld 3dim bounded ld number} implies Theorem \ref{thm: acc mld enc dim 3}. 

Nevertheless, in order to prove Theorem \ref{thm: acc mld enc dim 3}, we have to prove Theorem \ref{thm: acc mld 3dim bounded ld number}, and the proofs of these two theorems are intertwined with each other (see Sections 2 and 6 for details).

\medskip

\noindent\textbf{Further remarks and conjectures}. We remark that, in the proof of the ACC for mlds in dimension $2$ \cite{Ale93}, there are two cases: 

\medskip

\noindent\textbf{Case 1}. The dual graph of the minimal resolution is bounded.

\medskip

\noindent\textbf{Case 2}. The dual graph of the minimal resolution is unbounded.

\medskip

These two cases are treated in different ways in \cite{Ale93}. Note that the minimal resolution of a surface is nothing but its terminalization. Therefore, if we regard ``terminalization" as a kind of ``minimal resolution" in high dimensions, then Theorem \ref{thm: acc mld 3dim bounded ld number} implies that the ACC for mlds holds in dimension $3$ whenever the dual graph of the terminalization is bounded. In other words, \textbf{Case 1} in dimension $3$ is proved. Moreover, one may show a stronger result for \textbf{Case 1}: any log discrepancy which is not larger than $1$ belongs to a finite set (cf. \cite[Lemma A.2]{HL23b},\cite{HL24}). In the same fashion, we can also show this holds in dimension $3$. 

\begin{cor}\label{cor: acc ld 3dim bounded ld number}
Let $N$ be a non-negative integer, and $\Ii\subset [0,1]$ a DCC set. Then there exists an ACC set $\Ii'$ depending only on $N$ and $\Ii$ satisfying the following. Assume that $(X,B)$ is a klt pair of dimension $3$, such that 
\begin{enumerate}
    \item $\coeff(B)\subseteq\Ii$, and
    \item there are at most $N$ different (exceptional) log discrepancies of $(X,B)$ that are $\le 1$, i.e.
    $$\#(\{a(E,X,B)\mid E\text{ is exceptional over }X\}\cap [0,1])\le N,$$
\end{enumerate}
then $\{a(E,X,B)\mid E\text{ is exceptional over }X\}\cap [0,1]\subset\Ii'$.
\end{cor}

We propose two conjectures for enc pairs. Conjecture \ref{conj: index of enc} is related to the local K-stability theory (cf. \cite[Theorem 1.5, Conjecture 1.6]{HLQ23}). Roughly speaking, we expect that enc singularities have bounded local volumes (cf. \cite{LLX20}):

\begin{conj}[Local boundedness for enc pairs]\label{conj: index of enc}
Let $d$ be a positive integer, $\epsilon$ a positive real number, and $\Ii\subset[0,1]$ a DCC set. Then there exists a positive real number $\delta$ depending only on $d$ and $\Ii$ satisfying the following. Assume that $(X\ni x,B)$ is an enc germ of dimension $d$ such that $\coeff(B)\subseteq\Ii$. Then
\begin{enumerate}
    \item $(X\ni x,B)$ admits a $\delta$-plt blow-up,
    \item if $\mld (X\ni x,B)\ge \epsilon$, then  the local volume $\widehat\vol(X\ni x,B)$ is bounded away from $0$.
\end{enumerate}
\end{conj}

\begin{conj}[ACC for mlds of enc pairs]\label{conj: acc mld enc}
Let $d$ be a positive integer, and $\Ii\subset [0,1]$ a set of real numbers. Let
$$\textrm{eMLD}_{d}(\Ii):=\{\mld(X,B)\mid (X,B)\text{ is enc, }\dim X=d, \coeff(B)\subseteq\Ii\}.$$
\begin{enumerate}
    \item If $\Ii$ satisfies the DCC, then $\textrm{eMLD}_{d}(\Ii)$ satisfies the ACC. 
    \item If $\Ii$ is a finite set, then $\textrm{eMLD}_{d}(\Ii)$ is a discrete set away from $0$.
    \item[ (2')] If $\Ii$ is a finite set, then $\textrm{eMLD}_{d}(\Ii)$ satisfies the ACC.
    \end{enumerate}
\end{conj}

We refer the reader to \cite{Zhu24} on some related works on local K-stability theory and mlds.

\smallskip

Conjecture \ref{conj: acc mld enc}(2') is a weak form of Conjecture \ref{conj: acc mld enc}(2). We list it separately as we only need to assume it instead of Conjecture \ref{conj: acc mld enc}(2) in many results of this paper.

By \cite[Theorem 1.3]{HLS24}, Conjecture \ref{conj: index of enc} implies Conjecture \ref{conj: acc mld enc}. In particular, in dimension $4$, either Conjecture \ref{conj: index of enc} or \ref{conj: acc mld enc} implies the termination of flips by Theorem \ref{thm: enc and tof dim 4}. 

There is some evidence towards Conjectures \ref{conj: index of enc} and \ref{conj: acc mld enc}. We may prove both conjectures for surfaces. In dimension $3$, Conjecture \ref{conj: acc mld enc}(1) is nothing but Theorem \ref{thm: acc mld enc dim 3}, and we also have the following evidence.
\begin{enumerate}
    \item Let $\epsilon_0:=1-\sup\{t\mid t\in\CT(3,\Ii,\Zz_{\ge 1})\}$, where $\CT(3,\Ii,\Zz_{\ge 1})$ is a set of threefold canonical thresholds (see Theorem \ref{thm: 3fold acc ct}). Then $(X\ni x,B)$ admits a canonical blow-up which extracts the unique exceptional prime divisor computing $\mld(X\ni x,B)$ if $\mld(X\ni x,B)<\epsilon_0$. This proves a special case of Conjecture \ref{conj: index of enc}.
    \item To prove Conjecture \ref{conj: acc mld enc}(2), we are only left to prove the following two cases:
    \begin{enumerate}
        \item The case when the index $1$ cover of $X$ is strictly canonical.
        \item The case when $X$ is terminal.
    \end{enumerate}
    All other cases follow from our proofs in this paper. 
    \item When $\Ii=\{0\}$, since the mlds under case (2.a) belong to the set $\{\frac{1}{n}\mid n\in\Zz_{\ge 1}\}$ while (2.b) can never happen, we could get Conjecture \ref{conj: acc mld enc}(2). In addition, if we assume the index conjecture of Shokurov (cf. \cite[Conjecture 6.3]{CH21}), then we can get Conjecture \ref{conj: index of enc}(2) as well.
\end{enumerate}

Finally, we remark that enc singularities are deeply related to exceptional Fano varieties and Calabi-Yau varieties with small mlds or large indices via mirror symmetry. They are also tightly connected to the boundedness of log Calabi-Yau varieties. See Section \ref{sec: remark} for details.

\medskip

\noindent\textbf{Acknowledgement}. The authors would like to thank Paolo Cascini, Guodu Chen, Christopher D. Hacon, Yong Hu, Chen Jiang, Junpeng Jiao, Zhan Li, Yuchen Liu, Yujie Luo, Fanjun Meng, Lu Qi, V. V. Shokurov, Lingyao Xie, and Qingyuan Xue for useful discussions.  Liu also thanks useful discussions with Louis Esser, Burt Totaro, and Chengxi Wang on useful discussions on Remarks \ref{rem: enc and exceptional pair} and \ref{rem: enc and cy varieties} after the first version of the paper. 
Part of the work was done during the visit of the second author to the University of Utah in December 2021, and March and April 2022, and the second author would like to thank their hospitality.

Han is affiliated with LMNS at Fudan University. Han is supported by the National Key Research and Development Program of China (\#2023YFA1010600, \#2020YFA0713200), and NSFC for Excellent Young Scientists (\#12322102). Liu is supported by the National Key Research and Development Program of China (\#2023YFA1014400).

\section{Sketch of the proofs of theorems \ref{thm: acc mld enc dim 3} and \ref{thm: acc mld 3dim bounded ld number}}

Since the proofs of Theorems \ref{thm: acc mld enc dim 3} and \ref{thm: acc mld 3dim bounded ld number} are quite complicated, for the reader's convenience, we sketch a proof of them in this section. To prove Theorems \ref{thm: acc mld enc dim 3} and \ref{thm: acc mld 3dim bounded ld number}, we need to apply induction on the lower bound of mlds. More precisely, we need to prove the following theorems for positive integers $l$:

\medskip

\noindent {\bf \hypertarget{thm: E}{Theorem E$_l$}} (cf. Theorem \ref{thm: acc mld enc dim 3}){\bf.}
{\it 
Let $l$ be a positive integer, $\Ii\subset [0,1]$ a DCC set, and
$$\mathcal{E}(l,\Ii):=\left\{(X,B)\Bigm| (X,B)\text{ is }\Qq\text{-factorial enc,}\dim X=3, \coeff(B)\subseteq\Ii,\mld(X,B)>\frac{1}{l}\right\},$$
then $\{\mld(X,B)\mid (X,B)\in \mathcal{E}(l,\Ii)\}$ satisfies the ACC.}

\medskip

\noindent {\bf \hypertarget{thm: N}{Theorem N$_l$}} (cf. Theorem \ref{thm: acc mld 3dim bounded ld number}){\bf.} 
{\it Let $l$ and $N$ be positive integers, $\Ii\subset [0,1]$ a DCC set, and
\begin{align*}
    \mathcal{N}(l,N,\Ii):=\Bigg\{(X,B)\Biggm|
    \begin{array}{r@{}l}
        (X,B)\text{ is a threefold klt pair, } \coeff(B)\subseteq\Ii,\frac{1}{l}<\mld(X,B)<1,\\
        \#(\{a(E,X,B)\mid E\text{ is exceptional over }X\}\cap [0,1])\le N
    \end{array}\Bigg\},
    \end{align*}
then $\{\mld(X,B)\mid (X,B)\in  \mathcal{N}(l,N,\Ii)\}$ satisfies the ACC.}

\medskip

To make our proof more clear, we introduce the following auxiliary theorem: 

\medskip

\noindent {\bf \hypertarget{thm: C}{Theorem C$_l$}}{\bf.} 
{\it 
Let $l$ be a positive integer, $\Ii\subset [0,1]$ a DCC set, and
\begin{align*}
    \mathcal{C}(l,\Ii):=\Bigg\{(X,B)\Biggm|
    \begin{array}{r@{}l}
        (X\ni x,B)\text{ is a }\Qq\text{-factorial threefold enc germ, } \coeff(B)\subseteq\Ii,\\
       \mld(X)<1, \mld(X,B)>\frac{1}{l}, \text{$\tilde{X}\ni \tilde{x}$ is strictly canonical,}\\ \text{where  $\pi:(\tilde{X}\ni \tilde{x})\to (X\ni x)$ is the index }1\text{ cover of }X\ni x
    \end{array}\Bigg\},
    \end{align*}
    then $\{\mld(X,B)\mid (X,B)\in \mathcal{C}(l,\Ii)\}$ satisfies the ACC.} Here ``strictly canonical" means canonical but not terminal.

\medskip

We will prove Theorems  \hyperlink{thm: E}{E}, \hyperlink{thm: N}{N}, and \hyperlink{thm: C}{C}\footnote{Regarding the labels of the theorems: \hyperlink{thm: E}{E} stands for the initial of ``enc", \hyperlink{thm: N}{N} stands for the additional restriction ``$\le N$", and \hyperlink{thm: C}{C} stands for the initial of ``canonical" as in ``strictly canonical". Coincidentally, these labels together also form the word enc.}, that is, Theorems \hyperlink{thm: E}{E}$_l$, \hyperlink{thm: N}{N}$_l$,  and \hyperlink{thm: C}{C}$_l$ for any positive integer $l$, in the following way:
\begin{enumerate}
    \item Theorem \hyperlink{thm: E}{E}$_l$ implies Theorem \hyperlink{thm: N}{N}$_l$; cf. Lemma \ref{lem: e to n}.
    \item Theorem \hyperlink{thm: C}{C}$_l$ implies Theorem \hyperlink{thm: E}{E}$_l$; cf. Lemma \ref{lem: c to e}, and Theorems \ref{thm: acc mld enc dim 3 strictly canonical case},  \ref{thm: enc ACC fin coeff}, \ref{thm: enc terminal reduce to fin coeff}, \ref{thm: enc smooth cover case}, and \ref{thm: enc cDV cover case}.
    \item Theorem \hyperlink{thm: N}{N}$_{l-1}$ imply Theorem \hyperlink{thm: C}{C}$_l$; cf. Lemma \ref{lem: n to c}. 
\end{enumerate}

\medskip

The proof of (1) relies on the following observation: when proving by using contradiction, we may construct a DCC set $\Ii'$ depending only on $N$ and $\Ii$, such that for any $(X,B)\in\mathcal{N}(l,N,\Ii)$, there always exists $(X',B')\in\mathcal{E}(l,\Ii')$ such that $\mld(X',B')$ belongs to an ACC set if and only if $\mld(X,B)$ belongs to an ACC set. For such construction, one needs to look into different birational models of $(X,B)$ which only extracts non-canonical places of $(X,B)$. Indeed, the proof also works in higher dimensions and its idea is applied to prove Theorems \ref{thm: enc and tof}, \ref{thm: enc and tof dim 4}, and \ref{thm: strong lsc+acc to tof} as well. We refer to Lemma \ref{lem: reduce to enc} for more details.

\medskip

To prove (2), for any $(X,B)\in\mathcal{E}(l,\Ii)$, let $E$ be the unique exceptional prime divisor over $X$ such that $a(E,X,B)=\mld(X,B)$, and $x$ the generic point of $\Center_{X}(E)$. We may assume that $x$ is a closed point (see Lemma \ref{lem: acc enc center curve case}). There are three cases: 
\begin{enumerate}
    \item[(2.1)] $X\ni x$ is strictly canonical,
    \item[(2.2)] $X\ni x$ is not canonical, and
    \item[(2.3)] $X\ni x$ is terminal.
\end{enumerate}

We prove (2.1) by applying the cone theorem and the boundedness of complements (Theorem \ref{thm: acc mld enc dim 3 strictly canonical case}).

To prove (2.2), let $(\tilde X\ni \tilde{x})\rightarrow (X\ni x)$ be the index $1$ cover of $X\ni x$, then there are three sub-cases:
\begin{enumerate}
    \item[(2.2.1)] $\tilde X$ is strictly canonical,
    \item[(2.2.2)] $\tilde X$ is smooth, and
    \item[(2.2.3)] $\tilde X$ is terminal and has an isolated cDV singularity.
\end{enumerate}

The case (2.2.1) is nothing but Theorem \hyperlink{thm: C}{C}$_l$. 

In the case of (2.2.2), $X$ has toric singularities. Since $X$ is enc, one can show that the degree of the cover $\tilde X\rightarrow X$ is bounded from above, and the ACC for $\mld(X,B)$ immediately follows (see Theorem \ref{thm: enc smooth cover case}).

 We are left to prove (2.2.3), which is the most tedious part of the whole paper. Here we need to apply the classification of cDV singularities \cite{Mor85} to classify enc cDV (cyclic) quotient singularities $X\ni x$. Our ideas are inspired by the ones in the classification of threefold terminal singularities \cite{Mor85,Rei87} and the (rough) classification of  threefold ``nearly terminal" singularities \cite{Jia21,LX21,LL22}. The major difference between our classification and the previous ones is that the mlds of enc singularities cannot be assumed to be close to $1$: in fact, they can be arbitrarily small. This makes most computations in \cite{Rei87,Jia21,LX21,LL22} no longer work. A key observation here is that enc cDV quotient singularities with $\mld(X\ni x)\in (\frac{1}{l},\frac{1}{l-1}]$ share similar properties and will be much easier to classify. Indeed, we will show that the local Cartier indices of these $X\ni x$ are (almost) bounded. On the other hand, enc cDV quotient singularities with $\mld(X\ni x)\in (\frac{1}{l-1},1]$ can be classified by induction on $l$. This will imply (2.2.3), and we conclude the proof of (2.2). See Section 6.2 for more details.

The proof of (2.3) is very tricky. When $\Ii$ is a finite set, we can apply the theory of functional pairs introduced in \cite{HLS24,HLQ21} and carefully construct some weighted blow-ups to prove this case (Theorem \ref{thm: enc ACC fin coeff}). The key point is that the unique divisor $E$ over $X\ni x$ which computes $\mld(X,B)$ must also compute the canonical threshold $\ct(X,0;B)$, and the latter satisfies the ACC (\cite[Theorem 1.7]{HLL22}, \cite[Theorem 1.1]{Che22}). However, for the arbitrary DCC coefficient case, we are unable to prove it directly. Nevertheless, with some clever arguments, we can reduce it to the finite coefficient case (although possibly losing the condition that $X$ is terminal). See Theorem \ref{thm: enc terminal reduce to fin coeff} for more details. Since (2.3) is the only case left to prove in (2), the DCC coefficient case is also automatically resolved. This concludes the proof of (2.3), and hence of (2). 

\medskip

To prove (3), we may assume that $l\ge 2$. Let $K_{\tilde X}+\tilde B$ be the pullback of $K_X+B$, then we may show that
$$\mld(\tilde X,\tilde B)\in \{a(\tilde E,\tilde X,\tilde B)\le 1\mid \tilde E\text{ is exceptional over }\tilde X\}
\subseteq\{2\,\mld(X,B),\dots,(l-1)\,\mld(X,B)\}$$
and the latter is a finite set with cardinality $l-2$. Thus we only need to show that $\mld(\tilde X,\tilde B)$ belongs to an ACC set, which follows from Theorem \hyperlink{thm: N}{N}$_{l-1}$ as $\mld(\tilde X,\tilde B)\ge 2\,\mld(X,B)>\frac{2}{l}\ge\frac{1}{l-1}$. This finishes the proof of (3).

\medskip

To summarize, we may show Theorems \hyperlink{thm: E}{E}, \hyperlink{thm: N}{N}, and \hyperlink{thm: C}{C} by induction on $l$. Now Theorems \ref{thm: acc mld enc dim 3} and \ref{thm: acc mld 3dim bounded ld number} follow from Theorem \hyperlink{thm: N}{N} and \cite[Theorem 1.1]{HLL22}.

\medskip

The following flowchart may also help the reader. 

\begin{table}[ht]
\caption{Flowchart of the proofs of Theorems \ref{thm: acc mld enc dim 3} and \ref{thm: acc mld 3dim bounded ld number}}\label{tbl: flowchart}
\begin{adjustbox}{width=0.88\textwidth,right}
$\xymatrix{
*+[F]\txt{$(X\ni x,B),\mld>\frac{1}{l}$\\
$\coeff(B)\subseteq\Ii$}\ar@{->}[d] & & \\
*+[F]\txt{$(X\ni x,B)$ enc?}\ar@{->}[d]_{\text{Y}}\ar@{->}[r]^{\text{N}} & *+[F]\txt{Replace $(X\ni x,B)$,$\Ii$\\ $(X\ni x,B)\rightarrow$ enc\\ (\hyperlink{thm: E}{E}$_l$ to \hyperlink{thm: N}{N}$_l$: Lemmas \ref{lem: reduce to enc}, \ref{lem: e to n})}\ar@{->}[lu] & \\
*+[F]\txt{$X$ $\Qq$-factorial?}\ar@{->}[d]_{\text{Y}}\ar@{->}[r]^{\text{N}} & *+[F]\txt{Take $\Qq$-factorialization}\ar@/_2pc/[l] & \\
*+[F]\txt{$x$ closed?}\ar@{->}[d]_{\text{Y}}\ar@{->}[r]^{\text{N}} & *+[F]\txt{Surface case\checkmark\\ (Lemma \ref{lem: acc enc center curve case})} & \\
*+[F]\txt{$\mld(X)=1$?}\ar@{->}[d]_{\text{N}}\ar@{->}[r]^{\text{Y}} & *+[F]\txt{Cone theorem$+$complement\checkmark\\ (Theorem \ref{thm: acc mld enc dim 3 strictly canonical case})} & \\
*+[F]\txt{$X$ terminal?}\ar@{->}[d]_{\text{N}}\ar@{->}[r]^{\text{Y}} & *+[F]\txt{$\Ii$ finite?}\ar@{->}[dr]_{\text{Y}}\ar@{->}[r]^{\text{N}} &*+[F]\txt{Replace $(X\ni x,B)$ and $\Ii$\\
$\Ii\rightarrow$finite\\ (\hyperlink{thm: C}{C}$_l$ to \hyperlink{thm: E}{E}$_l$: Theorem \ref{thm: enc terminal reduce to fin coeff}, Lemma \ref{lem: c to e})}\ar@/_10pc/[lluuuuu]  \\
*+[F]\txt{$\tilde X$ smooth?}\ar@{->}[d]_{\text{N}}\ar@{->}[r]^{\text{Y}} & *+[F]\txt{Show boundedness of index \checkmark\\ (Theorem \ref{thm: enc smooth cover case})} & *+[F]\txt{Uniform rational polytope+\\
construct weighted blow-ups\checkmark\\ (Theorem \ref{thm: enc ACC fin coeff})}\\
*+[F]\txt{$\tilde X$ terminal?}\ar@{->}[d]_{\text{N}}\ar@{->}[r]^{\text{Y}} & *+[F]\txt{Classify enc singularities\checkmark\\ (Section 5.2)} & \\
*+[F]\txt{$\mld(\tilde X)=1$}\ar@{->}[d]& &\\
*+[F]\txt{Replace $(X\ni x,B)$ and $\Ii$\\
$l\rightarrow l-1$\\ (\hyperlink{thm: N}{N}$_{l-1}$ to \hyperlink{thm: C}{C}$_l$: Lemma \ref{lem: n to c})}\ar@/^11pc/[uuuuuuuuu]
}$ 
\end{adjustbox}
\end{table}

\section{Preliminaries}

We will freely use the notation and definitions from \cite{KM98,BCHM10}. 

\subsection{Pairs and singularities}

\begin{defn}\label{defn contraction}
A \emph{contraction} is a projective morphism $f: Y\rightarrow X$ such that $f_*\mathcal{O}_Y=\mathcal{O}_X$. In particular, $f$ is surjective and has connected fibers.
\end{defn}

\begin{defn}\label{defn: divisorial contraction}
Let $f: Y\rightarrow X$ be a birational morphism, and $\Exc(f)$ the exceptional locus of $f$. We say that $f$ is a \emph{divisorial contraction} of a prime divisor $E$ if $\Exc(f)=E$ and $-E$ is $f$-ample.
\end{defn}

\begin{defn}[Pairs, {cf. \cite[Definition 3.2]{CH21}}] \label{defn sing}
A \emph{pair} $(X/Z\ni z, B)$ consists of a contraction $\pi: X\rightarrow Z$, a (not necessarily closed) point $z\in Z$, and an $\mathbb{R}$-divisor $B\ge 0$ on $X$, such that $K_X+B$ is $\Rr$-Cartier over a neighborhood of $z$ and $\dim z<\dim X$. If $\pi$ is the identity map and $z=x$, then we may use $(X\ni x, B)$ instead of $(X/Z\ni z,B)$. In addition, if $B=0$, then we use $X\ni x$ instead of $(X\ni x,0)$. When we consider a pair $(X\ni x, \sum_{i} b_iB_i)$, where $B_i$ are distinct prime divisors and $b_i>0$, we always assume that $x\in \Supp B_i$ for each $i$.

If $(X\ni x,B)$ is a pair for any codimension $\ge 1$ point $x\in X$, then we call $(X,B)$ a pair. A pair $(X\ni x, B)$ is called a \emph{germ} if $x$ is a closed point. We also say $X\ni x$ is a \emph{singularity} if $X\ni x$ is a germ.
\end{defn}

\begin{defn}[Singularities of pairs]\label{defn: relative mld}
 Let $(X/Z\ni z,B)$ be a pair associated with the contraction $\pi: X\to Z$, and let $E$ be a prime divisor over $X$ such that $z\in\pi(\Center_X E)$. Let $f: Y\rightarrow X$ be a log resolution of $(X,B)$ such that $\Center_Y E$ is a divisor, and suppose that $K_Y+B_Y=f^*(K_X+B)$ over a neighborhood of $z$. We define $a(E,X,B):=1-\mult_EB_Y$ to be the \emph{log discrepancy} of $E$ with respect to $(X,B)$. 
 
 For any prime divisor $E$ over $X$, we say that $E$ is \emph{over} $X/Z\ni z$ if $\pi(\Center_X E)=\bar z$. If $\pi$ is the identity map and $z=x$, then we say that $E$
 is \emph{over} $X\ni x$. We define
 $$\mld(X/Z\ni z,B):=\inf\{a(E,X,B)\mid E\text{ is over }Z\ni z\}$$
 to be the \emph{minimal log discrepancy} (\emph{mld}) of $(X/Z\ni z,B)$.
 
 Let $\epsilon$ be a non-negative real number. We say that $(X/Z\ni z,B)$ is lc (resp. klt, $\epsilon$-lc,$\epsilon$-klt) if $\mld(X/Z\ni z,B)\ge 0$ (resp. $>0$, $\ge\epsilon$, $>\epsilon$). We say that $(X,B)$ is lc (resp. klt, $\epsilon$-lc, $\epsilon$-klt) if $(X\ni x,B)$ is lc (resp. klt, $\epsilon$-lc, $\epsilon$-klt) for any codimension $\ge 1$ point $x\in X$. 
 
 We say that $(X,B)$ is \emph{canonical} (resp. \emph{terminal}, \emph{plt}) if $(X\ni x,B)$ is $1$-lc (resp. $1$-klt, klt) for any codimension $\ge 2$ point $x\in X$.
 
 For any (not necessarily closed) point $x\in X$, we say that $(X,B)$ is lc (resp. klt, $\epsilon$-lc, canonical, terminal) near $x$ if $(X,B)$ is lc (resp. klt, $\epsilon$-lc, canonical, terminal) in a neighborhood of $x$. If $X$ is lc (resp. klt, $\epsilon$-lc, canonical, terminal) near a closed point $x$, then we say that $X\ni x$ is an lc (resp. klt, $\epsilon$-lc, canonical, terminal) singularity. We remark that if $(X\ni x, B)$ is lc, then $(X,B)$ is lc near $x$.
 
 We say that $(X\ni x,B)$ (resp. $(X,B)$) is \emph{strictly canonical} if $\mld(X\ni x,B)=1$ (resp. $\mld(X,B)=1$).
\end{defn}

\begin{defn}\label{defn: alct local}  Let $a$ be a non-negative real number, $(X\ni x,B)$ (resp. $(X,B)$) a pair, and $D\ge 0$ an $\Rr$-Cartier $\Rr$-divisor on $X$. We define
$$a\text{-}\lct(X\ni x,B;D):=\sup\{-\infty,t\mid t\ge 0, (X\ni x,B+tD)\text{ is }a\text{-}lc\}$$
$$\text{(resp. }a\text{-}\lct(X,B;D):=\sup\{-\infty,t\mid t\ge 0, (X,B+tD)\text{ is }a\text{-}lc\} \text{)}$$
to be the \emph{$a$-lc threshold} of $D$ with respect to $(X\ni x,B)$ (resp. $(X,B)$). We define
$$\ct(X\ni x,B;D):=1\text{-}\lct(X\ni x,B;D)$$
$$\text{(resp. }\ct(X,B;D):=\sup\{-\infty,t\mid t\ge 0, (X,B+tD)\text{ is canonical}\} \text{)}$$
to be the \emph{canonical threshold} of $D$ with respect to $(X\ni x,B)$ (resp. $(X,B)$). We define $$\lct(X\ni x,B;D):=0\text{-}\lct(X\ni x,B;D)$$ $$\text{(resp. } \lct(X,B;D):=0\text{-}\lct(X,B;D)\text{)}$$ to be the \emph{lc threshold} of $D$ with respect to $(X\ni x,B)$ (resp. $(X,B)$).
\end{defn}

\begin{thm}[{\cite[Theorem 1.6]{HLL22}}]\label{thm: alct acc terminal threefold}
Let $a\ge 1$ be a positive real number, and $\Ii\subset [0,1],~\Ii'\subset [0,+\infty)$ two DCC sets. Then the set of $a$-lc thresholds
   $$\{a\text{-}\lct(X\ni x,B;D)\mid \dim X=3, X\text{ is terminal}, \coeff(B)\subseteq\Ii, \coeff(D)\subseteq\Ii'\},$$
     satisfies the ACC.
\end{thm}

\begin{thm}[{\cite[Theorem 1.7]{HLL22}, \cite[Theorem 1.2]{Che22}}]\label{thm: 3fold acc ct}
Let $\Ii\subset[0,1]$ and $\Ii'\subset[0,+\infty)$ be two DCC sets. Then the set
$$\CT(3,\Ii,\Ii'):=\{\ct(X,B;D)\mid\dim X=3, \coeff(B)\subseteq\Ii,\coeff(D)\subseteq\Ii'\}$$
satisfies the ACC.
\end{thm}

\begin{defn}\label{defn: enc}\

\begin{enumerate}
    \item Let $(X,B)$ be a pair. We say that $(X,B)$ is \emph{exceptionally non-canonical} (\emph{enc} for short) if $\mld (X,B)<1$, and the set
$$\{E\mid E\text{ is exceptional over }X, a(E,X,B)\le 1\}$$
contains a unique element. 
\item Let $(X\ni x,B)$ be a germ. We say that $(X\ni x,B)$ is \emph{exceptionally non-canonical} (\emph{enc} for short) if $(X,B)$ is enc in a neighborhood of $x$ and $\mld(X\ni x,B)=\mld(X,B)$.
\end{enumerate}
\end{defn}
It is easy to see that any enc pair is automatically klt.
\begin{lem}
    Let $(X,B)$ be an enc pair. Then $(X,B)$ is klt.
\end{lem}
\begin{proof}
    Since $(X,B)$ is an enc pair, there exists an exceptional divisor over $X$. In particular, $\dim X\ge 2$. Let $f: Y\rightarrow X$ be a log resolution of $(X,\Supp B)$ and write $K_Y+B_Y:=f^*(K_X+B)$. If $(X,B)$ is not klt, then there exists a component $D$ of $\Supp B_Y$ such that $\mult_DB_Y\ge 1$. Let $H_1,H_2$ be two general hyperplane sections on $Y$. For each $i\in\{1,2\}$, let $y_i$ be the generic point of $H_i\cap D$ and let $E_i$ be the exceptional divisor obtained by the blow-up of $Y$ at $y_i$. Then for each $i\in\{1,2\}$, $E_i$ is exceptional over $X$ and $$1\ge 2-\mult_DB_Y\ge a(E_i,Y,B_Y)=a(E_i,X,B),$$
    which contradicts our assumption.
\end{proof}

\begin{thm}[{\cite[Theorem 18.22]{Kol+92}}]\label{thm: number of coefficients local}
Let $(X\ni x,\sum_{i=1}^m b_iB_i)$ be an lc germ such that $B_i\ge 0$ are $\Qq$-Cartier near $x$, $b_i\ge 0$, and $x\in\Supp B_i$ for each $i$. Then $\sum_{i=1}^m b_i\le \dim X$.
\end{thm}

\begin{defn}
Let $S$ be a set. We define $\#S$ or $|S|$ to be the cardinality of $S$.
\end{defn}

\subsection{Complements}\label{section: complements}
\begin{defn}\label{defn: complement}
Let $n$ be a positive integer, $\Ii_0\subset (0,1]$ a finite set, and $(X/Z\ni z,B)$ and $(X/Z\ni z,B^+)$ two pairs. We say that $(X/Z\ni z,B^+)$ is an \emph{$\Rr$-complement} of $(X/Z\ni z,B)$ if 
\begin{itemize}
    \item $(X/Z\ni z,B^+)$ is lc,
    \item $B^+\ge B$, and
    \item $K_X+B^+\sim_{\Rr}0$ over a neighborhood of $z$.
\end{itemize}
We say that $(X/Z\ni z,B^+)$ is an \emph{$n$-complement} of $(X/Z\ni z,B)$ if
\begin{itemize}
\item $(X/Z\ni z,B^+)$ is lc,
\item $nB^+\ge \lfloor (n+1)\{B\}\rfloor+n\lfloor B\rfloor$, and
\item $n(K_X+B^+)\sim 0$ over a neighborhood of $z$.
\end{itemize}
We say that $(X/Z\ni z,B)$ is $\Rr$-complementary (resp. $n$-complementary) if $(X/Z\ni z,B)$ has an $\Rr$-complement (resp. $n$-complement). 

 We say that $(X/Z\ni z,B^+)$ is a \emph{monotonic $n$-complement} of $(X/Z\ni z,B)$ if $(X/Z\ni z,B^+)$ is an $n$-complement of $(X/Z\ni z,B)$ and $B^+\ge B$.
 
 We say that $(X/Z\ni z,B^+)$ is an \emph{$(n,\Ii_0)$-decomposable $\Rr$-complement} of $(X/Z\ni z,B)$ if there exists a positive integer $k$, $a_1,\dots,a_k\in\Ii_0$, and $\Qq$-divisors $B_1^+,\dots,B_k^+$ on $X$, such that
\begin{itemize}
\item $\sum_{i=1}^ka_i=1$ and  $\sum_{i=1}^ka_iB_i^+=B^+$,
\item $(X/Z\ni z,B^+)$ is an $\Rr$-complement of $(X/Z\ni z,B)$, and
\item  $(X/Z\ni z,B_i^+)$ is an $n$-complement of itself for each $i$.
\end{itemize}
\end{defn}

\begin{thm}[{\cite[Theorem 1.10]{HLS24}}]\label{thm: ni decomposable complement}
Let $d$ be a positive integer and $\Ii\subset [0,1]$ a DCC set. Then there exists a positive integer $n$ and a finite set $\Ii_0\subset (0,1]$ depending only on $d$ and $\Ii$ and satisfying the following. 

Assume that $(X/Z\ni z,B)$ is a pair of dimension $d$ and $\coeff(B)\subseteq\Ii$, such that $X$ is of Fano type over $Z$ and $(X/Z\ni z,B)$ is $\Rr$-complementary. Then $(X/Z\ni z,B)$ has an $(n,\Ii_0)$-decomposable $\Rr$-complement. 
\end{thm}

\subsection{Threefold singularities}

\begin{lem}[{\cite[Lemma~5.1]{Kaw88}}]\label{lem: local index is local cartier index}
Let $X\ni x$ be a terminal threefold singularity and $I$ a positive integer such that $IK_X$ is Cartier near $x$. Then $ID$ is Cartier near $x$ for any $\Qq$-Cartier Weil divisor $D$ on $X$.
\end{lem}

\begin{thm}[{\cite[Theorem 1.4]{LX21}, \cite[Theorem 1.3]{Jia21}}]\label{thm: 12/13} Let $X$ be a $\Qq$-Gorenstein threefold. If $\mld(X)<1$, then $\mld(X)\le\frac{12}{13}$.\end{thm}

\begin{defn}\label{defn: weights of monomials and poly} A \emph{weight vector} is a vector $w\in\mathbb Q_{>0}^d$ for some positive integer $d$.  

For any vector $\bm{\alpha}=(\alpha_1,\dots, \alpha_d)\in \Zz_{\ge 0}^d$, we define $\bm{x}^{\bm{\alpha}} :=x_1^{\alpha_1}\cdots x_d^{\alpha_d}$, and $w(\bm{x}^{\bm{\alpha}}):=\sum_{i=1}^dw_i\alpha_i$ to be \emph{the weight of $\bm{x}^{\bm{\alpha}}$ with respect to $w$}.
For any analytic function $0\neq h:=\sum_{\bm{\alpha}\in \Zz^d_{\ge 0}} a_{\bm{\alpha}}\bm{x}^{\bm{\alpha}}$, we define $w(h):=\min\{w(\bm{x}^{\bm{\alpha}})\mid a_{\bm{\alpha}}\neq 0\}$ to be \emph{the weight of $h$ with respect to $w$}. If $h=0$, then we define $w(h):=+\infty$.
\end{defn}

\begin{defn}\label{defn: weighted blowup log discrepancies}
Let $(X\ni x,B:=\sum_{i=1}^k b_iB_i)$ be a threefold germ such that $X$ is terminal, $b_i\ge 0$, and $B_i\ge 0$ are $\Qq$-Cartier Weil divisors. Let $d,n$, and $m<d$ be positive integers, such that
$$(X\ni x)\cong (\phi_1=\cdots =\phi_m=0)\subset (\mathbb C^d\ni o)/\frac{1}{n}(a_1,\dots,a_d)$$ for some non-negative integers $a_1,\dots,a_d$ and some semi-invariant irreducible analytic function $\phi_1\dots,\phi_m\in \mathbb{C}\{x_1,\dots,x_d\}$ such that $\mult_o \phi_i>1$ for each $i$, and the group action on $\Cc^d$ corresponding to $\frac{1}{n}(a_1,\dots,a_d)$ is free outside $o$. By \cite[Lemma~5.1]{Kaw88}, $B_i$ can be identified with $\big((h_i=0)\subset \mathbb (\mathbb{C}^d\ni o)/\frac{1}{n}(a_1,\dots,a_d)\big)|_X$ for some semi-invariant analytic function $h_i\in \mathbb{C}\{x_1,\dots,x_d\}$ near $x\in X$. We say that $B_i$ is locally defined by $(h_i=0)$ for simplicity. The set of \emph{admissible weight vectors} of $X\ni x$ is defined by $$\left\{\frac{1}{n}(w_1,\dots,w_d)\in\frac{1}{n}\mathbb Z^d_{>0}\Bigm| \text{there exists $b\in\mathbb{Z}$ such that $w_i\equiv b a_i~\mathrm{mod}~n$, $1\le i\le d$}\right\}.$$ For any admissible weight vector $w=\frac{1}{n}(w_1,\dots,w_d)$, 
we define 
$$w(X\ni x):=\frac{1}{n}\sum_{i=1}^d w_i-\sum_{i=1}^m w(\phi_i)-1, \text{ and } w(B):=\sum_{i=1}^k b_iw(h_i).$$
By construction, $w(B)$ is independent of the choices of $b_i$ and $B_i$.

Let $f': W\to (\Cc^d\ni o)/\frac{1}{n}(a_1,\dots,a_d)$ be the weighted blow-up at $o$ with an admissible weight vector $w:=\frac{1}{n}(w_1,\dots,w_d)$, $Y$ the strict transform of $X$ on $W$, and $E'$ the exceptional locus of $f'$. Let $f:=f'|_{Y}$ and $E:=E'|_{Y}$. We say that $f: Y\to X\ni x$ is a \emph{weighted blow-up} at $x\in X$ and $E$ is the exceptional divisor of this weighted blow-up.
\end{defn}

We will use the following well-known lemma frequently.
\begin{lem}[cf. {\cite[the proof of Theorem 2]{Mor85}, \cite[\S 3.9]{Hay99}}]\label{lem: weighted blowup log discrepancies}
Settings as in Definition~\ref{defn: weighted blowup log discrepancies}. For any admissible weight vector $w$ of $X\ni x$, let $E$ be the exceptional divisor of the corresponding weighted blow-up $f: Y\rightarrow X$ at $x$ (cf. Definition~\ref{defn: weighted blowup log discrepancies}). If $E$ is a prime divisor, then
$$K_Y=f^*K_X+w(X\ni x)E, \text{ and }~f^*B=B_Y+w(B)E,$$
where $B_Y$ is the strict transform of $B$ on $Y$. In particular,  $a(E,X,B)=1+w(X\ni x)-w(B)$.
\end{lem}

\section{On termination of flips}

\subsection{Proofs of Theorems \ref{thm: enc and tof} and \ref{thm: enc and tof dim 4}} We need the following auxiliary lemma for induction purposes.
\begin{lem}\label{lem: high dimension imply low dimension}
Let $d\ge 2$ be a positive integer. We have the following.
\begin{enumerate}
    \item The ACC for (global) mlds of enc pairs with finite coefficients in dimension $d$ implies the ACC for (global) mlds of enc pairs with finite coefficients in dimension $\le d$.
    \item The termination of $\mathbb Q$-factorial terminal (resp. klt, lc) flips in dimension $d$ implies the termination of $\mathbb Q$-factorial terminal (resp. klt, lc) flips in dimension $\le d$.
\end{enumerate}
\end{lem}

\begin{proof}
(1) It suffices to show that for any enc pair $(X,B)$, $(X',B'):=(X\times\Cc,B\times\Cc)$ is also an enc pair, and $\mld(X',B')=\mld(X,B)$. Let $E$ be the unique exceptional prime divisor over $X$, such that $a(E,X,B)=\mld (X,B)<1$. By \cite[Proposition 2.36]{KM98}, there exists a log resolution $f:Y\to X$ of $(X,B)$, such that $\Supp B_Y^{> 0}$ is log smooth, where $K_Y+B_Y:=f^{*}(K_X+B)$. Then $\mult_{E}B_Y=1-a(E,X,B)>0$. Let $Y':=Y\times \Cc$, $B_{Y'}:=B_Y\times \Cc$, $E'=E\times \Cc$, and $f':=f\times \id_{\mathbb C}$. Then $f': Y'\to X'$ is a log resolution of $(X',B')$, 
$$K_{Y'}+B_{Y'}={f'}^{*}(K_{X'}+B'),$$
$\Supp B_{Y'}^{>0}$ is log smooth, and $\mult_{E'}B_{Y'}=\mult_{E}B_Y>0$. 

Since $(X,B)$ is enc, for any point $y'$ on $Y'$ such that $\overline{y'}\neq E'$ is exceptional over $X'$, we have $\codim y'>\mult_{y'}B_{Y'}+1$. Thus by \cite[Lemma 3.3]{CH21}, $$\mld(Y'\ni y',B_{Y'})=\codim y'-\mult_{y'}B_{Y'}>1.$$ In other words, $E'$ is the unique exceptional prime divisor over $X'$ such that $a(E',X',B')\le 1$. Hence $(X',B')$ is an enc pair, and $\mld(X',B')=a(E',X',B')=a(E,X,B)=\mld(X,B)$. 

(2) Let $(X,B)$ be a $\mathbb Q$-factorial terminal (resp. klt, lc) pair, and $$(X,B):=(X_0,B_0)\dashrightarrow (X_1,B_1)\dashrightarrow\cdots (X_i,B_i)\dashrightarrow\cdots$$ a sequence of flips. Let $C$ be an elliptic curve. Then
$$(X\times C,B\times C):=(X_0\times C,B_0\times C)\dashrightarrow (X_1\times C,B_1\times C)\dashrightarrow\cdots (X_i\times C,B_i\times C)\dashrightarrow\cdots$$ 
is also a sequence of flips of dimension $\dim X+1$. Now (2) follows from our assumptions.
\end{proof}

We will use the following notation in the proofs of Lemma \ref{lem: reduce to enc} and Theorem \ref{thm: enc and tof}.

\begin{defn}
Let $(X,B)$ be an lc pair. We define
$$D(X,B)_{\le 1}:=\{E\mid E\text{ is exceptional over }X, a(E,X,B)\le 1\}.$$
By \cite[Proposition 2.36]{KM98}, $D(X,B)$ is a finite set when $(X,B)$ is klt.
\end{defn}

The following lemma plays a key role in this section, and it will be applied to prove Theorems \ref{thm: enc and tof}, \ref{thm: reduction to enc acc mld}, and Lemma \ref{lem: e to n}.

\begin{lem}\label{lem: reduce to enc}
Let $d,N$ be two positive integers, and $\Ii\subset [0,1]$ a DCC set. Let $\{(X_i,B_i)\}_{i=1}^{\infty}$ be a sequence of klt pairs of dimension $d$, and $$\Ii_i:=\{a(E_i,X_i,B_i)\mid E_i\text{ is exceptional over }X_i\}\cap [0,1].$$ Suppose that
\begin{itemize}
    \item $\coeff(B_i)\subseteq\Ii$ for each $i$,
    \item $1\le \#\Ii_i\le N$ for each $i$, and
    \item $\cup_{i=1}^{\infty}\Ii_i$ does not satisfy the ACC.
\end{itemize}
Then possibly passing to a subsequence, there exists a DCC set $\Ii'\subset [0,1]$, and a sequence $\{(X_i',B_i')\}_{i=1}^{\infty}$ of $\Qq$-factorial enc pairs of dimension $d$, such that
\begin{enumerate}
    \item $\coeff(B_i')\subseteq\Ii'$ for each $i$,
    \item $\{\mld(X_i',B_i')\}_{i=1}^{\infty}$ is strictly increasing, and
    \item $\mld(X_i',B_i')\ge\mld(X_i,B_i)$ for each $i$.
\end{enumerate}
Moreover, if we further assume that $\Ii$ is a finite set and $\cup_{i=1}^{\infty}\Ii_i$ is a DCC set, then we may choose $\Ii'$ to be a finite set.
\end{lem}
\begin{proof}
Possibly passing to a subsequence we may assume that $\#\Ii_i=k\ge 1$ for some positive integer $k\le N$. For each $i$, there exist positive integers $r_{1,i},\dots,r_{k,i}$ and real numbers $\{a_{i,j}\}_{i\ge 1,1\le j\le k}$, such that  $$D(X_i,B_i)_{\le 1}=\{E_{i,1,1},\dots,E_{i,1,r_{1,i}};E_{i,2,1},\dots,E_{i,2,r_{2,i}};\dots;E_{i,k,1},\dots,E_{i,k,r_{k,i}}\}$$ 
for some distinct exceptional prime divisors $E_{i,j,l}$ over $X_i$,  $a_{i,j}=a(E_{i,j,l},X_i,B_i)\le 1$ for any $i,j,l$, and $\{a_{i,j}\}_{i\ge 1,1\le j\le k}$ does not satisfy the ACC. Possibly reordering indices and passing to a subsequence, we may assume that there exists $1\le j_0\le k$, such that
\begin{itemize}
    \item $a_{i,j}$ is strictly increasing for any $1\le j\le j_0$, and
    \item $a_{i,j}$ is decreasing for any $j_0+1\le j\le k$.
\end{itemize}
Let $f_i: Y_i\rightarrow X_i$ be a birational morphism which extracts exactly the set of divisors $$\mathcal{F}_i:=\{E_{i,j,l}\}_{j_0+1\le j\le k,1\le l\le r_{j,i}}$$ and let $K_{Y_i}+B_{Y_i}:=f_i^*(K_{X_i}+B_i)$. Since $a_{i,j}$ is decreasing for any $j_0+1\le j\le k$, the coefficients of $B_{Y_i}$ belong to the DCC set $\tilde\Ii:=\Ii\cup\{1-a_{i,j}\}_{j_0+1\le j\le k}$. By construction, $D(Y_i,B_{Y_i})_{\le 1}=D(X_i,B_i)\backslash\mathcal{F}_i$, and $D(Y_i,B_{Y_i})_{\le 1}$ is a non-empty set as $E_{i,1,1}\in D(Y_i,B_{Y_i})_{\le 1}$.

Let $a_j:=\lim_{i\rightarrow+\infty}a_{i,j}$ for any $1\le j\le j_0$. By \cite[Lemma 5.3]{Liu18}, for each $i$, there exist $1\le j_i\le j_0$ and $1\le l_i\le r_{j_i,i}$, and a birational morphism $g_i: X_i'\rightarrow Y_i$ which extracts exactly all divisors in $D(Y_i,B_{Y_i})_{\le 1}$ except $E_{i,j_i,l_i}$, such that $X_i'$ is $\Qq$-factorial and
$$\sum_{1\le j\le j_0,1\le l\le r_{j,i},(j,l)\not=(j_i,l_i)}(a_j-a_{i,j})\mult_{E_{i,j_i,l_i}}E_{i,j,l}<a_{j_i}-a_{i,j_i}.$$
Let $B_{X_i'}:=(g_i^{-1})_*B_{Y_i}+\sum_{1\le j\le j_0,1\le l\le r_{j,i},(j,l)\not=(j_i,l_i)}(1-a_j)E_{i,j,l}$. Then the coefficients of $B_{X_i'}$ belong to the DCC set $\Ii':=\tilde\Ii\cup\{1-a_j\}_{1\le j\le j_0}$. Moreover, if $\Ii$ is a finite set and $\cup_{i=1}^{\infty}\Ii_i$ is a DCC set, then $\Ii'$ is a finite set. 

By construction, $(X_i',B_{X_i'})$ is enc and $\mld(X_i,B_i)\le a_{i,j_i}\le a(E_{i,j_i,l_i},X_i',B_{X_i'})<a_{j_i}\le 1$. Thus 
$$a_{i,j_i}\le a(E_{i,j_i,l_i},X_i',B_{X_i'})=\mld(X_i',B_{X_i'})<a_{j_i}\le 1.$$
Possibly passing to a subsequence, we may assume that $j_i=j_1$ is a constant. 
Since $a_{i,j_1}$ is strictly increasing for each $i$, $\{\mld(X_i',B_{X_i'})\mid i\in\Zz_{\ge 1}\}$ is not a finite set. Possibly passing to a subsequence, we may assume that $\{\mld(X_i',B_{X_i'})\}_{i=1}^{\infty}$ is strictly increasing. 
\end{proof}

\begin{proof}[Proof of Theorem \ref{thm: enc and tof}]
First, we prove the case of klt flips in dimension $d$. Let
 \begin{center}$
 \xymatrixrowsep{0.135in}
\xymatrixcolsep{0.09in}
\xymatrix{
(X,B):=& (X_0,B_0)\ar@{-->}[rr]\ar@{->}[dr] & &  (X_1,B_1)\ar@{-->}[rr]\ar@{->}[dl] & &  \cdots\ar@{-->}[rr] & &(X_i,B_i)\ar@{-->}[rr]\ar@{->}[dr] & & (X_{i+1},B_{i+1})\ar@{-->}[rr]\ar@{->}[dl] & & \cdots.\\
& & Z_0 &  & & & &  & Z_i & & & \\
}$
\end{center}
be a sequence of klt flips of dimension $d$. Since $a(E,X_i,B_i)\le a(E,X_j,B_j)$ for any $i\le j$ and any prime divisor $E$ over $X$, possibly truncating to a subsequence, we may assume that there exist exceptional prime divisors $E_1,\dots,E_k$ over $X$, such that $D(X_i,B_i)_{\le 1}=D(X,B)_{\le 1}=\{E_1,\dots,E_k\}$ for any $i$. Then $\{a(E_l,X_i,B_i)\mid 1\le l\le k\}_{i=1}^{\infty}$ is a DCC set and the coefficients of $B_i$ belong to a finite set. By Lemma \ref{lem: reduce to enc} and the ACC for (global) mlds of enc pairs with finite coefficients, $\{a(E_l,X_i,B_i)\mid 1\le l\le k\}_{i=1}^{\infty}$ satisfies the ACC. It follows that $\{a(E_l,X_i,B_i)\mid 1\le l\le k\}_{i=1}^{\infty}$ is a finite set. Thus possibly truncating to a subsequence, we may assume that $a(E_l,X_i,B_i)=a(E_l,X_j,B_j)$ for any $i,j$. 

Let $f_0: Y_0\rightarrow X_0$ be the birational morphism which extracts exactly $E_1,\dots,E_k$ for each $i$ such that $Y_0$ is $\mathbb Q$-factorial, and let $K_{Y_0}+B_{Y_0}:=f_0^*(K_{X_0}+B_0)$. We claim that we may construct a sequence of $\mathbb Q$-factorial terminal flips on $K_{Y_0}+B_{Y_0}$. Suppose that we have constructed $(Y_i,B_{Y_i})$ with $f_i:Y_i\to X_i$ such that $K_{Y_i}+B_{Y_i}:=f_i^*(K_{X_i}+B_i)$, and $(Y_0,B_{Y_0})\dashrightarrow (Y_1,B_{Y_1})\cdots \dashrightarrow  (Y_i,B_{Y_i})$ consists of a sequence of terminal flips on $K_{Y_0}+B_{Y_0}$. By \cite[Corollary 1.4.3]{BCHM10}, we may run a $(K_{Y_i}+B_{Y_i})$-MMP over $Z_i$ which terminates with a minimal model $(Y_{i+1},B_{Y_{i+1}})$. Since $(X_{i+1},B_{i+1})$ is the log canonical model of $(Y_i,B_{Y_i})$ over $Z_i$, there exists an induced morphism $f_{i+1}: Y_{i+1}\to X_{i+1}$, such that  $K_{Y_{i+1}}+B_{Y_{i+1}}:=f_{i+1}^*(K_{X_{i+1}}+B_{i+1})$. Since $a(E_l,X_i,B_i)=a(E_l,X_{i+1},B_{i+1})$ for any $l$, $E_l$ is not contracted in the MMP $Y_i\dashrightarrow Y_{i+1}$, and $Y_i\dashrightarrow Y_{i+1}$ only consists of a sequence of flips. Thus we finish the proof of the claim by induction. Now the termination follows from the termination of flips for $\Qq$-factorial terminal pairs.

Finally, we prove the general case. By Lemma \ref{lem: high dimension imply low dimension}, assumptions in Theorem \ref{thm: enc and tof} also hold for $\le d-1$. So we may do induction on $d$, and assume that any sequence of lc flips in dimension $\le d-1$ terminates. Now the termination follows from
the klt case and the special termination (cf. \cite[Corollary 4]{Sho04}, \cite{Fuj07}, \cite[Lemma 2.17(1)]{CT23}, \cite{HL22}).
\end{proof}

\begin{proof}[Proof of Theorem \ref{thm: enc and tof dim 4}] This follows from Theorem \ref{thm: enc and tof} and the termination of canonical fourfold flips \cite{Fuj04,Fuj05}. Note that \cite{Fuj04,Fuj05} only deals with the $\Qq$-coefficients case but the same argument works for the $\Rr$-coefficients case. 
\end{proof}

\subsection{Proof of Theorem \ref{thm: strong lsc+acc to tof}}  

\begin{lem}\label{lem: finite mld of a fixed pair}
Let $(X,B)$ be a pair. Then the set 
\[\{ \mld (X\ni x, B) \mid x\in X\}\]
is finite. In particular, $\min_{x\in S}\mld (X\ni x, B)$ is well-defined for any subset $S$ of $X$.
\end{lem}
\begin{proof}
Let $f:Y\to X$ be a log resolution of $(X,B)$, and $K_Y+B_Y:=f^{*}(K_X+B)$. Then 
\[\{ \mld (X\ni x, B) \mid  x\in X\}\subseteq \{ \mld (Y\ni y, B_Y) \mid  y\in Y\},\]
and the latter is a finite set by \cite[Lemma 3.3]{CH21}.
\end{proof}

\begin{defn}[{cf. \cite[Examples 19.1.3--19.1.6]{Ful98}, \cite[Ch.II.(4.1.5)]{Kol96}}]\label{defn:cycle}
Let $X$ be a reduced projective scheme and let $k$ be a non-negative integer. 
We denote by $Z_k (X)_{\Qq}$ the group of $k$-dimensional algebraic cycles on $X$ with rational coefficients. 
All cycles which are numerically equivalent to zero form a subgroup of $Z_k(X) _{\mathbb{Q}}$, 
and we denote by $N_k (X)_{\Qq}$ the quotient group. Then $N_k (X)_{\Qq}$ is a finite-dimensional $\Qq$-vector space. 
\end{defn}

\begin{lem}\label{lem:cycle}
Let $k$ be a positive integer, and $f: X \dashrightarrow Y$ a dominant rational map of reduced projective schemes. Suppose that $f$ induces a birational map on each irreducible component of $X$ and $Y$, and $f^{-1}$ does not contract any $k$-dimensional subvariety of $Y$. Then:
\begin{enumerate}
\item $\dim N_k(X)_{\Qq} \ge \dim N_k(Y)_{\Qq}$, and
\item if $f$ contracts some $k$-dimensional subvariety of $X$, then $\dim N_k(X)_{\Qq} > \dim N_k(Y)_{\Qq}$.
\end{enumerate}
\end{lem}

\begin{proof}
(1) This follows from the fact that $f_*: N_{k}(X)_{\Qq} \to N_{k}(Y)_{\Qq}$ is surjective.

(2) Let $W \subset X$ be a subvariety of dimension $k$ which is contracted by $f$. Then the cycle $[W]$ satisfies $[W] \not \equiv 0$ in $Z_k(X)_{\Qq}$, $f_* [W] \equiv 0$ in $Z_k(Y)_{\Qq}$, and (2) is proved. 
\end{proof}

Shokurov proved that the ACC and the LSC conjectures for mlds imply the termination of flips \cite{Sho04}. The following slightly stronger result actually follows from similar arguments as his proof. For the reader's convenience, we give a proof in details here.
\begin{thm}\label{thm: Strong ACC LSC imply termination}
Let $d$ be a positive integer, $a$ a non-negative real number, and $\Ii\subset[0,1]$ a finite set. Suppose that 
\begin{enumerate}
    \item the set of mlds
    $$\{\mld(X\ni x,B)\mid \mld(X,B)>a \text{ (resp.} \ge a), \dim X=d, \coeff(B)\subseteq\Ii\}$$
    satisfies the ACC, and
    \item for any pair $(X,B)$ of dimension $d$ such that $\mld(X,B)>a$ (resp. $\ge a$),
    $$x\rightarrow\mld(X\ni x,B)$$
    is lower-semicontinuous for closed points $x$. 
\end{enumerate}
Then for any pair $(X,B)$ with $\dim X=d$ and $\mld(X,B)>a$ (resp. $\ge a$), any sequence of $(K_X+B)$-flips terminates.
\end{thm}

\begin{proof}
\noindent\textbf{Step 1}. In this step, we introduce some notation. Suppose that there exists an infinite sequence of $(K_X+B)$-flips,
\[
(X,B):= (X_0,B_0) \overset{f_0}{\dashrightarrow} (X_1,B_1) \overset{f_1}{\dashrightarrow}  \cdots 
(X_{i-1},B_{i-1})\overset{f_{i-1}}{\dashrightarrow} (X_i,B_i) \overset{f_{i}}{\dashrightarrow}(X_{i+1},B_{i+1})  \cdots.
\]
For each $i\ge 0$, denote by $\phi_i: X_i\to Z_i$ and $\phi_i^+: X_{i+1}\to Z_i$ the corresponding flip contraction and flipped contraction between quasi-projective normal varieties respectively. Let $a_i:=\min_{x_i\in\mathrm{Exc}(\phi_i)}\mld(X_i\ni x_i,B_i)$, and $\alpha_i:=\inf\{a_j\mid j\ge i\}$. There exists an exceptional prime divisor $E_i$ over $X_i$, such that $a_i=a(E_i,X_i,B_i)$, and $\Center_{X_i}E_i\subseteq \mathrm{Exc}(\phi_i)$. We note that $a_i>a$ (resp. $a_i\ge a$) as $\mld(X_i,B_i)\ge \mld(X,B)$ and $\codim \mathrm{Exc}(\phi_i)\ge 2$. 

\medskip

\noindent\textbf{Step 2}. In this step, for each $i\ge 0$, we show that $\alpha_i=a_{n_i}$ for some $n_i\ge i$. 

For each $i\ge 0$, let $\alpha_{i}^l:=\min\{a_j\mid i\le j\le l\}$ for $l\ge i$. For each $l\ge i$, there exist $i\le i_l\le l$ and an exceptional prime divisor $E_{i_l}$ over $X_{i_l}$, such that $a_{i_l}=\alpha_{i}^l=a(E_{i_l},X_{i_l},B_{i_l})$, and $\Center_{X_{i_l}}E_{i_l}\subseteq \mathrm{Exc}(\phi_{i_l})$. Suppose that $\alpha_i\neq \alpha_{i}^l$ for any $l\ge i$, then there are infinitely many $l\ge i+1$, such that $i_l>i_{l-1}$. For each such $l$, and any $i\le j<i_l$, we have $\Center_{X_j} E_{i_l}\notin \mathrm{Exc}(\phi_j)$, otherwise $a(E_{i_l},X_j,B_j)< a(E_{i_l},X_{i_l},B_{i_l})=\alpha_i^l\le a_j$, which contradicts the definition of $a_j$. Thus $\mld(X_i\ni x_{i,i_l},B_i)=a(E_{i_l},X_i,B_i)=a(E_{i_l}, X_{i_l},B_{i_l})=a_{i_l}$, where $x_{i,i_l}=\Center_{X_i} E_{i_l}$. In particular, $\{\mld(X_i\ni x_{i,i_l},B_i)\mid l\ge i\}$ is an infinite set which contradicts Lemma \ref{lem: finite mld of a fixed pair}. Thus $\alpha_i=a_{n_i}$ for some $n_i\ge i$.

\medskip

\noindent \textbf{Step 3}. In this step, we show that possibly passing to a subsequence of flips, there exists a non-negative real number $a'>a$ (resp. $a'\ge a$), such that \begin{itemize}
\item $a_i \ge a'$ for any $i \ge 0$, and 
\item $a_i = a'$ for infinitely many $i$. 
\end{itemize}

Since $\mld (X_i,B_i)\ge \mld (X,B)>a$ (resp. $\mld (X_i,B_i)\ge \mld (X,B)\ge a$) for any $i$, by \textbf{Step 2} and (1), 
$$\{\alpha_i\mid i\in\Zz_{\ge 0}\}=\{a_{n_i}\mid i\in\Zz_{\ge 0}\}$$
is a finite set. In particular, there exist a non-negative integer $N$, and a unique non-negative real number $a'> a$ (resp. $a'\ge a$), such that $a_{n_i} \ge a'$ for any $i \ge N$, and $a_{n_i} = a'$ for infinitely many $i$.

\medskip

\noindent \textbf{Step 4}. In this step, we construct $k,S_i,W_i$, and show some properties. 

Possibly passing to a subsequence, we may assume there exists a non-negative integer $k$ satisfying the following. 
\begin{itemize}
\item For any $i$, any point $x_i \in \mathrm{Exc}(\phi_i)$ with $(X_i\ni x_i, B_i) = a'$ satisfies $\dim x_i \le k$. 
\item For infinitely many $i$, there exists a $k$-dimensional point $x_i \in \mathrm{Exc}(\phi_i)$ such that $(X_i\ni x_i, B_i) = a'$. 
\end{itemize}

Let $S_i$ be the set of the $k$-dimensional points $x_i \in X_i$ with 
$\mld(X_i\ni x_i, B_i) \le a'$, and $W_i \subset X_i$ the Zariski closure of $S_i$. Then by (2) and \cite[Proposition 2.1]{Amb99}, any $k$-dimensional point $x_i \in W_i$ belongs to $S_i$. 

\medskip

\noindent \textbf{Step 5}. In this step, we prove that $f_i$ induces 
\begin{itemize}
\item a bijective map $S_i \setminus \mathrm{Exc}(\phi_i) \to S_{i+1}$, and 
\item a dominant morphism $f_i' : W_i \setminus \mathrm{Exc}(\phi_i) \to W_{i+1}$. 
\end{itemize}

It suffices to show the first assertion as the second one follows from the first one. 

For any $x_i \in S_i \setminus  \mathrm{Exc}(\phi_i)$, $f_i(x_i) \in S_{i+1}$ as 
\[
\mld (X_{i+1}\ni f_i(x_i),B_{i+1}) = \mld (X_i\ni x_i, B_i) \le a'. 
\]
For any $x_{i+1} \in S_{i+1}$, suppose that $x_{i+1} \in \mathrm{Exc}(\phi_i^{+})$, then  
\[
\min_{x_i\in \mathrm{Exc}(\phi_i)}\mld(X_i\ni x_i, B_i) < \mld(X_{i+1}\ni x_{i+1},B_{i+1}) \le a', 
\]
which contradicts \textbf{Step 3}. It follows that $x_{i+1} \notin \mathrm{Exc}(\phi_i^{+})$ and $f_i$ induces a bijective map $S_i \setminus \mathrm{Exc}(\phi_i) \to S_{i+1}$.

\medskip

\noindent\textbf{Step 6}. In this step, we derive a contradiction, and finish the proof.

By \textbf{Step 5}, the number of the irreducible components of $W_i$ is non-increasing. Thus possibly passing to a subsequence, we may assume that 
$f_i$ induces a birational map on each irreducible component of $W_i$. On the one hand, by \textbf{Step 5} and \textbf{Step 4}, $f_i ^{-1}$ does not contract any $k$-dimensional subvariety of $W_{i+1}$. On the other hand, by construction of $k$ in \textbf{Step 4}, there exist infinitely many $i$ such that 
$\mld(X_i\ni x_i, B_i) = a'$ for some $k$-dimensional point $x_i \in \mathrm{Exc}(\phi_i)$. For such $i$ and $x_i$, by \textbf{Step 5}, $x _i \in W_i$ is contracted by $f_i$, which contradicts Lemma \ref{lem:cycle}. 
\end{proof}

\begin{thm}[$=$Theorem \ref{thm: strong lsc+acc to tof}]\label{thm: ter acc lsc+enc to tof}
Let $d$ be a positive integer. Assume that 
\begin{enumerate}
    \item the ACC for mlds of terminal pairs with finite coefficients in dimension $d$, i.e.
    $$\{\mld(X\ni x,B)\mid \mld(X,B)>1, \dim X=d, \coeff(B)\subseteq\Ii\}$$
    satisfies the ACC for any finite set $\Ii$, and
    \item the LSC for mlds of terminal pairs in dimension $d$, i.e. for any pair $(X,B)$ of dimension $d$ such that $\mld(X,B)>1$,
    $$x\rightarrow\mld(X\ni x,B)$$
    is lower-semicontinuous for closed points $x$.
\end{enumerate}
Then any sequence of terminal flips in dimension $d$ terminates. Moreover, if we additionally assume that Conjecture \ref{conj: acc mld enc}(2') holds in dimension $d$, then any sequence of lc flips in dimension $\le d$ terminates. 
\end{thm}

\begin{proof}
This follows from Theorem \ref{thm: Strong ACC LSC imply termination} when $a=1$ and Theorem \ref{thm: enc and tof}.
\end{proof}

\begin{rem}

Generalized pairs, introduced in \cite{BZ16}, have become central topics in birational geometry in recent years. By \cite{HL23a}, we can run MMPs for any $\Qq$-factorial lc generalized pair. Therefore, studying the termination of flips for generalized pairs is also intriguing. It is important to note that the proofs in this section are expected to work for generalized pairs as well. For instance, \cite[Theorem 4.8]{CGN24} provides a proof of Theorem \ref{thm: Strong ACC LSC imply termination} for generalized pairs when $a=0$. Consequently, we anticipate that Theorems \ref{thm: enc and tof}, \ref{thm: enc and tof dim 4}, and \ref{thm: strong lsc+acc to tof} will also apply to generalized pairs by using similar arguments to those in this section. For the sake of brevity and the reader's convenience, we omit the detailed proofs here.
\end{rem}

\section{Theorems \ref{thm: acc mld enc dim 3} and \ref{thm: acc mld 3dim bounded ld number} for canonical threefolds}

In this section, we prove Theorem \ref{thm: acc mld enc dim 3} when $X$ is canonical.

We first prove Theorem \ref{thm: acc mld enc dim 3} when $(X,B)$ is non-canonical in codimension $2$. More precisely, we have:

\begin{lem}\label{lem: acc enc center curve case}
Let $\Ii\subset[0,1]$ be a DCC set. Assume that $(X,B)$ is an enc pair of dimension $3$ and $E$ a prime divisor over $X$, such that
\begin{enumerate}
\item $\coeff(B)\subseteq\Ii$,
\item $\dim\Center_XE=1$, and
\item $a(E,X,B)\le 1$.
\end{enumerate}
Then $\mld(X,B)$ belongs to an ACC set.
\end{lem}
\begin{proof}
Since $(X,B)$ is enc, $a(E,X,B)=\mld(X,B)$. Let $C:=\Center_XE$, $H$ a general hyperplane on $X$ which intersects $C$, and $K_H+B_H:=(K_X+B+H)|_H$. Then the coefficients of $B_H$ belong to a DCC set. By \cite[Lemma 5.17(1)]{KM98}, $$\mld(H,B_H)\ge\mld(X,B+H)\ge a(E,X,B+H)=a(E,X,B)=\mld(X,B).$$ 
By \cite[Corollary 1.4.5]{BCHM10}, $a(E,X,B+H)\ge\mld(H,B_H)$. Thus $\mld(H,B_H)=\mld(X,B)$. By \cite[Theorem 3.8]{Ale93}, $\mld(H,B_H)$ belongs to an ACC set, hence $\mld(X,B)$ belongs to an ACC set.
\end{proof}

\subsection{Strictly canonical threefolds} In this subsection, we prove Theorem \ref{thm: acc mld enc dim 3} when $X$ is strictly canonical. More precisely, we have:

\begin{thm}\label{thm: acc mld enc dim 3 strictly canonical case}
Let $\Ii\subset [0,1]$ be a DCC set. Then
$$\{\mld(X\ni x,B)\mid \dim X=3, (X\ni x,B)\text{ is enc}, X\text{ is strictly canonical}, \coeff(B)\subseteq\Ii\}$$
satisfies the ACC.
\end{thm}

\begin{proof}
Suppose that the statement does not hold. Then there exists a sequence of enc pairs $(X_i\ni x_i,B_i)$ of dimension $3$, such that
\begin{itemize}
    \item $\mld(X_i)=1$, and
    \item $a_i:=\mld(X_i\ni x_i,B_i)$ is strictly increasing.
\end{itemize}
Possibly taking a small $\Qq$-factorialization, we may assume that $X_i$ is $\Qq$-factorial. By Lemma \ref{lem: acc enc center curve case}, we may assume that $x_i$ is a closed point for each $i$. 

Let $E_i$ be the unique prime divisor that is exceptional over $X_i$ such that $a(E_i,X_i,B_i)\le 1$. Then $a(E_i,X_i,B_i)=\mld(X_i\ni x_i,B_i)=a_i$ and $\Center_{X_i}E_i=x_i$. For any prime divisor $F_i\not=E_i$ that is exceptional over $X_i$, $a(F_i,X_i,0)\ge a(F_i,X_i,B_i)>1$. Since $X_i$ is strictly canonical, $a(E_i,X_i,0)=1$.

By Theorem \ref{thm: ni decomposable complement}, there exists a positive integer $n$ and a finite set $\Ii_0\subset (0,1]$ depending only on $\Ii$, such that for any $i$, there exists an $(n,\Ii_0)$-decomposable $\Rr$-complement $(X_i\ni x_i,B_i^+)$ of $(X_i\ni x_i,B_i)$. Possibly passing to a subsequence, we may assume that there exists a positive real number $a$, such that $a(E_i,X_i,B_i^+)=a$ for any $i$. Since $a_i$ is strictly increasing and $a=a(E_i,X_i,B_i^+)\le a(E_i,X_i,B_i)=a_i$, possibly passing to a subsequence, we may assume that there exists a positive real number $\delta$, such that $a_i-a>\delta$ for any $i$. 

Let $f_i: Y_i\rightarrow X_i$ be the divisorial contraction which extracts $E_i$, and let $B_{Y_i}$ and $B_{Y_i}^+$ be the strict transforms of $B_i$ and $B_i^+$ on $Y_i$ respectively. Then
$$K_{Y_i}+B_{Y_i}+(1-a_i)E_i=f_i^*(K_{X_i}+B_i).$$
By the length of extremal rays (cf. \cite[Theorem~4.5.2(5)]{Fuj17}), there exists a $(K_{Y_i}+B_{Y_i}+(1-a^+)E_i)$-negative extremal ray $R_i$ over a neighborhood of $x_i$ which is generated by a rational curve $C_i$, such that 
$$0>(K_{Y_i}+B_{Y_i}+(1-a)E_i)\cdot C_i\ge -6.$$
Since $(K_{Y_i}+B_{Y_i}+(1-a_i)E_i)\cdot C_i=0$, we have
$$0<(a_i-a)(-E_i\cdot C_i)\le 6.$$
Thus
$$0<(-E_i\cdot C_i)<\frac{6}{\delta}.$$
By \cite[Theorem 1.1]{Kaw15a}, $60K_{Y_i}$ is Cartier over a neighborhood of $x_i$. Since $X_i$ is enc, $Y_i$ is terminal. By Lemma \ref{lem: local index is local cartier index}, $60D_i$ is Cartier over a neighborhood of $x_i$ for any Weil divisor $D_i$ on $Y_i$. In particular, $-E_i\cdot C_i$ belongs to the finite set $\frac{1}{60}\mathbb Z_{\ge 1}\cap (0,\frac{6}{\delta})$.

We may write $B_i=\sum_{j} b_{i,j}B_{i,j}$ where $B_{i,j}$ are the irreducible components of $B_i$, and let $B_{Y_{i},j}$ be the strict transform of $B_{i,j}$ on $Y_i$ for each $i,j$. Then $B_{Y_{i},j}\cdot C\in\frac{1}{60}\mathbb Z_{\ge 0}$ for every $i,j$. Since $K_{Y_i}=f_i^*K_{X_i}$, $K_{Y_i}\cdot C_i=0$. Thus
$$a(E_i,X_i,B_i)=1-\frac{(K_{Y_i}+B_{Y_i})\cdot C_i}{(-E_i\cdot C_i)}=1-\sum_j b_{i,j}\frac{(B_{Y_i,j}\cdot C_i)}{(-E_i\cdot C_i)}.$$
Since $b_{i,j}\in\Ii$, $a(E_i,X_i,B_i)$ belongs to an ACC set, this leads to a contradiction. 
\end{proof}

\subsection{Terminal threefolds}

In this subsection, we study Theorem \ref{thm: acc mld enc dim 3} when $X$ is terminal. At the moment, we cannot prove Theorem \ref{thm: acc mld enc dim 3} in full generality, but we can prove the finite coefficient case, and reduce the DCC coefficient case to the finite coefficient case (but possibly losing the condition that $X$ is terminal).

\begin{lem}\label{lem: acc mld exceptionally non-canoncial bdd sing}
Let $I$ be a positive integer and $\Ii\subset [0,1]$ be a DCC set. Assume that $(X\ni x,B)$ is a threefold pair, such that
\begin{enumerate}
    \item $X$ is terminal, 
    \item $\coeff(B)\subseteq\Ii$,
    \item $(X\ni x,B)$ is enc, and
    \item $IK_X$ is Cartier near $x$. 
\end{enumerate}Then $\mld(X\ni x,B)$ belongs to an ACC set.
\end{lem}
\begin{proof}
Possibly replacing $X$ with a small $\Qq$-factorialization, we may assume that $X$ is $\Qq$-factorial. By Lemma \ref{lem: acc enc center curve case}, we may assume that $x$ is a closed point.

Suppose that the statement does not hold. By Theorem \ref{thm: number of coefficients local}, there exist a non-negative integer $m$, a real number $a\in (0,1]$, a strictly increasing sequence of real numbers $a_i\in (0,1]$ such that $\lim_{i\rightarrow +\infty}a_i=a$, and a sequence of $\Qq$-factorial threefold germs $(X_i\ni x_i,B_i=\sum_{j=1}^mb_{i,j}B_{i,j})$, such that for any $i$,
\begin{itemize}
\item $(X_i\ni x_i,B_i)$ is enc and $X_i$ is terminal,
\item $b_{i,j}\in\Ii$ for any $j$,
\item $\mld(X_i\ni x_i,B_i)=a_i$,
\item $B_{i,1},\ldots,B_{i,m}$ are the irreducible components of $B_i$,
\item $x_i\in\Supp B_{i,j}$ for any $j$, and
\item $IK_{X_i}$ is Cartier near $x_i$.
\end{itemize}

Possibly passing to a subsequence, we may assume that $b_{i,j}$ is increasing for any fixed $j$, and let $b_j:=\lim_{i\rightarrow+\infty}b_{i,j}$. We let $\bar B_i:=\sum_{j=1}^mb_jB_{i,j}$ for each $i$. By \cite[Theorem 1.1]{HMX14}, possibly passing to a subsequence, we may assume that $(X_i\ni x_i,\bar B_i)$ is lc for each $i$.

Since  $(X_i\ni x_i,B_i)$ is enc, we may let $E_i$ be the unique prime divisor over $X_i\ni x_i$ which computes $\mld(X_i\ni x_i,B_i)$. By Lemma \ref{lem: local index is local cartier index} and \cite[Theorem 1.2]{Nak16}, possibly passing to a subsequence, we may assume that $a':=a(E_i,X_i,\bar B_i)\ge 0$ is a constant, and we may pick a strictly decreasing sequence of real numbers $\epsilon_i$, such that $(1+\epsilon_i)B_i\ge\bar B_i\not=B_i$ for each $i$ and $\lim_{i\rightarrow+\infty}\epsilon_i=0$. Therefore,
$$a'=a(E_i,X_i,\bar B_i)<a(E_i,X_i,B_i)=a_i<a,$$ and
$$\lim_{i\to +\infty} \epsilon_i\mult_{E_i}B_i\ge \lim_{i\to +\infty}(a(E_i,X_i,B_i)-a(E_i,X_i,\bar B_i))=a-a'>0.$$ Hence $\lim_{i\to +\infty}\mult_{E_i}B_i= +\infty$. 
Let $t_i:=\ct(X_i\ni x_i, 0;B_i)$. Since $(X_i\ni x_i,B_i)$ is enc and $a(E_i,X_i,B_i)<1$, we have that $t_i<1$. By \cite[Lemma 2.12(1)]{HLL22}, $a(E_i,X_i,t_iB_i)=1$ for each $i$. Since $$1=a(E_i,X_i,t_iB_i)=a(E_i,X_i,B_i)-(1-t_i)\mult_{E_i}B_i< a-(1-t_i)\mult_{E_i}B_i,$$
we have 
$$1>t_i>1-\frac{a-1}{\mult_{E_i}B_i}.$$
Thus $\lim_{i\rightarrow+\infty}t_i=1$ and $t_i<1$ for each $i$, which contradicts Theorem \ref{thm: alct acc terminal threefold}.
\end{proof}

\begin{defn}\label{defn: terminal blow-up}
Let $(X\ni x,B)$ be an lc germ. A \emph{terminal blow-up} of $(X\ni x,B)$ is a birational morphism $f: Y\rightarrow X$ which extracts a prime divisor $E$ over $X\ni x$, such that $a(E,X,B)=\mld(X\ni x,B)$, $-E$ is $f$-ample, and $Y$ is terminal.
\end{defn}

\begin{lem}[{\cite[Lemma 2.35]{HLL22}}]\label{lem: can extract divisor computing ct that is terminal strong version}
Let $(X\ni x, B)$ be a germ such that $X$ is terminal and $\mld(X\ni x, B)=1$. Then there exists a terminal blow-up $f: Y\to X$ of $(X\ni x,B)$. Moreover, if $X$ is $\Qq$-factorial, then $Y$ is $\Qq$-factorial.
\end{lem}

\begin{thm}\label{thm: enc ACC fin coeff}
Let $\Ii\subset [0,1]$ be a finite set. Then
$$\{\mld(X\ni x,B)\mid \dim X=3, (X\ni x,B)\text{ is enc}, X\text{ is terminal, } \coeff(B)\subseteq\Ii\}$$
satisfies the ACC.
\end{thm}
\begin{proof}

\noindent\textbf{Step 1}. We construct some functional pairs in this step.

\smallskip

Possibly replacing $X$ with a small $\Qq$-factorialization, we may assume that $X$ is $\Qq$-factorial. By Lemma \ref{lem: acc enc center curve case}, we may assume that $x$ is a closed point. By Lemmas \ref{lem: acc enc center curve case} and \ref{lem: acc mld exceptionally non-canoncial bdd sing}, we may assume that $X\ni x$ is a $cA/n$ type singularity for some positive integer $n\ge 3$. Let $E$ be the unique prime divisor over $X\ni x$ such that $\mld(X\ni x,B)=a(E,X,B)$. 

Let $t:=\ct(X\ni x,0;B)$. By Theorem \ref{thm: alct acc terminal threefold}, there exists a positive real number $\epsilon_0$ depending only on $\Ii$, such that $t\le 1-\epsilon_0$. 

By our assumptions and \cite[Theorem 5.6]{HLS24}, there exist two positive integers $l,m$, real numbers $1,v^1_0,\dots,v^l_0$ that are linearly independent over $\Qq$, $\bm{v_0}:=(v^1_0,\dots,v^l_0)\in\mathbb R^l$, an open set $U\ni\bm{v_0}$ of $\mathbb R^l$, and $\Qq$-linear functions $s_1,\dots,s_m:\mathbb R^{l+1}\rightarrow\Rr$ depending only on $\Ii$, and Weil divisors $B^1,\dots,B^m\ge 0$ on $X$, such that:
\begin{enumerate}
\item $s_i(1,\bm{v_0})>0$ and $x\in\Supp B^i$ for each $1\le i\le m$.
\item let $B(\bm{v}):=\sum_{i=1}^m s_i(1,\bm{v})B^i$ for any $\bm{v}\in\mathbb R^l$. Then $B(\bm{v}_0)=B$ and $(X\ni x,B(\bm{v}))$ is lc for any $\bm{v}\in U$.
\item Possibly shrinking $U$, we may assume that $(1+\frac{\epsilon_0}{2})B\ge B(\bm{v})\ge (1-\frac{\epsilon_0}{2})B$ for any $\bm{v}\in U$.
\end{enumerate}
We may pick vectors $\bm{v}_1,\dots,\bm{v}_{l+1}\in U\cap\mathbb Q^l$ and real numbers $b_1,\dots,b_{l+1}\in (0,1]$ depending only on $\Ii$, such that $\sum_{i=1}^{l+1}b_i=1$ and $\sum_{i=1}^{l+1}b_i\bm{v}_i=\bm{v}_0$. We let $B_i:=B(\bm{v}_i)$ for each $i$. Then there exists a positive integer $M$ depending only on $\Ii$, such that $MB_i$ is integral for any $i$.

Let $t_i:=\ct(X\ni x,0;B_i)$ for each $i$. By \cite[Lemma 2.12(2)]{HLL22}, $\mld(X\ni x,t_iB_i)=1$. Since $B_i\ge (1-\frac{\epsilon_0}{2})B$ and $t\le 1-\epsilon_0$, we have 
$$t_i\le \frac{1-\epsilon_0}{1-\frac{\epsilon_0}{2}}<1.$$ 
Thus
$$B\ge \frac{1-\epsilon_0}{1-\frac{\epsilon_0}{2}}\left(1+\frac{\epsilon_0}{2}\right)B\ge\frac{1-\epsilon_0}{1-\frac{\epsilon_0}{2}}B_i\ge t_iB_i.$$
Since $(X\ni x,B)$ is enc, $a(E,X,t_iB_i)=1$ and $a(F,X,t_iB_i)>1$ for any prime divisor $F\not=E$ over $X\ni x$.

\medskip

\noindent\textbf{Step 2}. Construct divisorial contractions.

\smallskip

By Lemma~\ref{lem: can extract divisor computing ct that is terminal strong version}, there exists a terminal blow-up $f: Y\rightarrow X$ of $(X\ni x, t_iB_i)$ which extracts a prime divisor $E$ over $X\ni x$. By \cite[Theorem 1.3]{Kaw05}, $f$ is of ordinary type as $n\ge 3$. By \cite[Theorem 2.31(1)]{HLL22}, we may take suitable local coordinates $x_1,x_2,x_3,x_4$ of $\mathbb C^4$, an analytic function $\phi\in\mathbb C\{x_1,x_2,x_3,x_4\}$, and positive integers $r_1,r_2,a,b$ and $d$ satisfying
\begin{itemize}
\item $\gcd(b,n)=1$,
    \item $(X\ni x)\cong (\phi(x_1,x_2,x_3,x_4)=0)\subset(\mathbb C^4\ni o)/\frac{1}{n}(1,-1,b,0)$,
    \item $\phi$ is semi-invariant under the group action
    $$\bm{\mu}:=(x_1,x_2,x_3,x_4)\rightarrow(\xi x,\xi^{-1}x_2,\xi^{b}x_3,x_4),$$
    where $\xi=e^{\frac{2\pi i}{n}}$,
    \item $f$ is a weighted blow-up at $x\in X$ with the weight vector $w:=\frac{1}{n}(r_1,r_2,a,n)$,
    \item $a\equiv br_1\mod n$, $\gcd(\frac{a-br_1}{n},r_1)=1$, $r_1+r_2=adn$,
    \item $\phi(x_1,x_2,x_3,x_4)=x_1x_2+g(x_3^n,x_4)$, and
    \item $z^{dn}\in g(x_3^n,x_4)$ and $w(\phi)=adn$.
\end{itemize}
There are two cases: 
\begin{itemize}
    \item[\textbf{Case 1}.]  $d\ge 4$ or $a\ge 4$.
    \item[\textbf{Case 2}.] $d\le 3$ and $a\le 3$.
\end{itemize}

\medskip

\noindent\textbf{Step 3}.  In this step we deal with \textbf{Case 1}, that is, the case when $d\ge 4$ or $a\ge 4$. In this case, we can pick three positive integers $r_1',r'_2$ and $a'$, such that
\begin{itemize}
    \item $r'_1+r'_2=a'dn$,
    \item $a'\equiv br'_1\mod n$,
    \item $r'_1,r'_2>n$, and
    \item $\frac{1}{n}(r'_1,r'_2,a',n)\neq \frac{1}{n}(r_1,r_2,a,n)$.
\end{itemize}
In fact, when $a\ge 4$, we may take $a'=3$. When $d\ge 4$, we may take $a'=1$ and $(r'_1,r'_2)\neq (r_1,r_2)$. 

Let $w':=\frac{1}{n}(r'_1,r'_2,a',n)$. Since $a\ge a'$, by \cite[Lemma C.7]{HLL22}, the weighted blow-up with the weight vector $w'$ at $x\in X$ under analytic local coordinates $x_1,x_2,x_3,x_4$ extracts 
an exceptional prime divisor $E'\not=E$, such that $w'(X\ni x)=\frac{a'}{n}$. By our assumptions,
$$a(E',X,B)=1+\frac{a'}{n}-\sum_{i=1}^{l+1}b_iw'(B_i)>1,$$
hence $\sum_{i=1}^{l+1}b_iw'(B_i)<\frac{a'}{n}$. Since $MB_i$ is integral for each $i$ and $x\in\Supp B_i$, $w'(B_i)\in\frac{1}{Mn}\mathbb Z_{\ge 1}$ for each $i$. 

Let $b_0:=\min\{b_i\mid i=1,2,\ldots,l+1\}$. By Theorem \ref{lem: acc mld exceptionally non-canoncial bdd sing}, we may assume that $n>\frac{3M}{\gamma_0}$. Since $MB_i$ a Weil divisor, $MB_i=(h_i=0)\subset\mathbb C^4/\frac{1}{n}(1,-1,b,0)$ for some analytic function $h_i$. Since  $w'(x_1)>1,w'(x_2)>1,w'(x_4)=1$, $a'\le 3$, $b_i\ge b_0$, and $n>\frac{3M}{\gamma_0}$, for each $i$, there exists a positive integer $1\le p_i<\frac{Ma'}{\gamma_0}$, such that $x_3^{p_i}\in h_i$ and $w'(x_3^{p_i})=w'(h_i)=\frac{p_ia'}{n}$, and
$$\sum_{i=1}^{l+1}b_i\frac{p_ia'}{Mn}=\sum_{i=1}^{l+1}b_iw'(B_i)<\frac{a'}{n}.$$
In particular, $\sum_{i=1}^{l+1}\frac{b_ip_i}{M}<1$. We have
$$w(B)=\sum_{i=1}^{l+1}b_iw(B_i)=\sum_{i=1}^{l+1}\frac{b_i}{M}w(h_i)\le \sum_{i=1}^{l+1}\frac{b_i}{M}w(x_3^{p_i})= \sum_{i=1}^{l+1}\frac{b_ip_i}{M}\cdot\frac{a}{n}<\frac{a}{n},$$
hence $a(E,X,B)=a(E,X,0)-\mult_EB=1+\frac{a}{n}-w(B)>1$, a contradiction.

\medskip

\noindent\textbf{Step 4}.  In this step, we deal with \textbf{Case 2}, that is, the case when $d\le 3$ and $a\le 3$, hence we conclude the proof. In this case, since $a\equiv br_1\mod n$ and $\gcd(b,n)=1$, $\gcd(r_1,n)=\gcd(a,n)\le a\le 3$, so $\gcd(r_1,n)\mid 6$. Since $r_1+r_2=adn$, $$\gcd(r_1,r_2)=\gcd(r_1,adn)\mid ad\gcd(r_1,n)\mid 216.$$
Since $MB_i$ is a Weil divisor, $nMB_i$ is Cartier near $x$. Thus
$$\frac{a}{n}=a(E,X,0)-a(E,X,t_iB_i)=t_i\mult_EB_i=\frac{t_i}{nM}\mult_EnMB_i\in\frac{t_i}{nM}\mathbb Z_{\ge 1},$$
which implies that $\frac{aM}{t_i}\in\mathbb Z_{\ge 1}$. Then $\frac{aM}{t_i}(K_X+t_iB_i)$ is a Weil divisor for any $i$, by \cite[Lemma 5.3]{HLL22}, $\frac{216aM}{t_i}(K_X+t_iB_i)$ is Cartier near $x$.

By \cite[4.8 Corollary]{Sho94a}, there exists a prime divisor $E'\not=E$ over $X\ni x$ such that $a(E',X,0)=1+\frac{a'}{n}$ for some integer $a'\in\{1,2\}$. Since $\frac{216aM}{t_i}(K_X+t_iB_i)$ is Cartier near $x$ and $a(E',X,t_iB_i)>1$ for any $i$, there exists a positive integer $k_i$ such that
$$a(E',X,t_iB_i)=1+\frac{k_it_i}{216aM}.$$
Thus
$$\mult_{E'}B_i=\frac{1}{t_i}(a(E',X,0)-a(E',X,t_iB_i))=\frac{a'}{nt_i}-\frac{k_i}{216aM}.$$
Since $MnB_i$ is Cartier near $x$ and $x\in\Supp B_i$, there exist positive integers $a_i'$, such that
$\mult_{E'}B_i=\frac{a_i'}{Mn}$.
Since $(X\ni x,B=\sum_{i=1}^{l+1}b_iB_i)$ is enc and $a(E,X,B)<1$, $a(E',X,B)>1$, hence
\begin{equation}\label{equ: non-canonical fin coeff equ 1}
\frac{\sum_{i=1}^{l+1}b_ia_i'}{Mn}=\mult_FB<\frac{a'}{n}.
\end{equation}
Thus the $a_i'$ belong to a finite set depending only on $\Ii$.

Since $a(E,X,0)=1+\frac{a}{n}$ and $a(E,X,t_iB_i)=1$, $\mult_EB_i=\frac{a}{nt_i}$. Thus $$a(E,X,B_i)=a(E,X,0)-\mult_EB_i=1+\frac{a}{n}-\frac{a}{nt_i}.$$ Since $(X\ni x,B_i)$ is lc, $a(E,X,B_i)>0$, so $$nt_i>\frac{a}{1+\frac{a}{n}}=\frac{1}{\frac{1}{a}+\frac{1}{n}}>\frac{1}{2},$$
and we have $\frac{a'}{nt_i}<2a'$. Since
\begin{equation}\label{equ: non-canonical fin coeff equ 2}
\frac{a'}{nt_i}-\frac{k_i}{216aM}=\mult_{E'}B_i=\frac{a_i'}{nM}
\end{equation}
and each $k_i$ is a positive integer, the $k_i$ belong to a finite set depending only on $\Ii$. By \eqref{equ: non-canonical fin coeff equ 1} and \eqref{equ: non-canonical fin coeff equ 2}, we have
 $$a'\left(\sum_{i=1}^{l+1}\frac{b_i}{nt_i}-\frac{1}{n}\right)-\sum_{i=1}^{l+1}\frac{k_ib_i}{216Ma}=  \sum_{i=1}^{l+1}\left(\frac{b_ia'}{nt_i}-\frac{b_ik_i}{216Ma}\right)-\frac{a'}{n}=\left(\sum_{i=1}^{l+1}\frac{b_ia_i'}{Mn}\right)-\frac{a'}{n}<0. $$
Since $a',k_i,b_i,M,a,a_i'$ belong to a finite set depending only on $\Ii$, $$\sum_{i=1}^{l+1}\frac{b_ia_i'}{Mn}-\frac{a'}{n}=-\frac{1}{n}\left(a'-\sum_{i=1}^{l+1}\frac{b_ia_i'}{M}\right)$$
belongs to a DCC set depending only on $\Ii$, and
$\sum_{i=1}^{l+1}\frac{b_i}{nt_i}-\frac{1}{n}$ also belongs to a DCC set depending only on $\Ii$. Since
$$\mult_EB_i=\frac{1}{t_i}(a(E,X,0)-a(E,X,t_iB_i))=\frac{1}{t_i}(1+\frac{a}{n}-1)=\frac{a}{nt_i},$$
we have
\begin{align*}
   \mld(X\ni x,B)&=a(E,X,B)=a(E,X,0)-\mult_EB=a(E,X,0)-\sum_{i=1}^{l+1}b_i\mult_EB_i\\
   &=1+\frac{a}{n}-\sum_{i=1}^{l+1}\frac{ab_i}{nt_i}=1-a\left(\sum_{i=1}^{l+1}\frac{b_i}{nt_i}-\frac{1}{n}\right).
\end{align*}
Thus $\mld(X\ni x,B)$ belongs to an ACC set depending only on $\Ii$.
\end{proof}

\begin{thm}\label{thm: enc terminal reduce to fin coeff}
Let $\epsilon\in (0,1)$ be a positive real number. Suppose that
$$\{\mld(X\ni x,B)\mid \dim X=3, \coeff(B)\subseteq\Ii_0, (X\ni x,B)\text{ is }\Qq\text{-factorial enc}\}\cap [\epsilon,1]$$
satisfies the ACC for any finite set $\Ii_0\subset [0,1]$. Then for any DCC set $\Ii\subset [0,1]$,
$$\{\mld(X\ni x,B)\mid \dim X=3, X\text{ is terminal}, \coeff(B)\subseteq\Ii, (X\ni x,B)\text{ is enc}\}\cap [\epsilon,1]$$
satisfies the ACC.
\end{thm}
\begin{proof}
Possibly replacing $X$ with a small $\Qq$-factorialization, we may assume that $X$ is $\Qq$-factorial. Suppose that the statement does not hold. By Theorem \ref{thm: number of coefficients local}, there exist a positive integer $m$, a real number $a$, a strictly increasing sequence of real numbers $a_i$, and a sequence of $\Qq$-factorial enc threefold pairs $(X_i\ni x_i,B_i=\sum_{j=1}^m b_{i,j}B_{i,j})$, such that for any $i$,
\begin{itemize}
\item $b_{i,j}\in\Ii$, and $B_{i,j}\ge 0$ are Weil divisors for any $j$,
\item for any fixed $j$, $b_{i,j}$ is increasing,
\item $\mld(X_i\ni x_i,B_i)=a_i$, and
\item $\lim_{i\rightarrow+\infty}a_i=a\in (\epsilon, 1]$.
\end{itemize}
By \cite[Theorem 1.1]{HLL22}, $a<1$. Let $b_j:=\lim_{i\rightarrow +\infty}b_{i,j}$ and $\bar B_i:=\sum_{j=1}^m b_jB_{i,j}$. By \cite[Theorem 1.1]{HMX14}, possibly passing to a subsequence, we may assume that $(X_i\ni x_i,\bar B_i)$ is lc for each $i$. Let $t_i:=\ct(X_i\ni x_i;B_i)$ and $E_i$ the unique prime divisor over $X_i\ni x_i$ such that $a(E_i,X_i,B_i)<1$. Then $a(E_i,X_i,B_i)=a_i$. By \cite[Lemma 2.12(1)]{HLL22},  $a(E_i,X_i,t_iB_i)=1$, hence $\mult_{E_i}B_i=\frac{1-a_i}{1-t_i}<\frac{1}{1-t_i}$. By construction, $t_i<1$ for each $i$. By Theorem \ref{thm: alct acc terminal threefold}, we may assume that $t_i$ is decreasing, hence there exists a positive real number $M$ such that $\mult_{E_i}B_i<M$. 

By construction, there exists a sequence of positive real numbers $\epsilon_i$ such that $(1+\epsilon_i)B_i\ge\bar B_i$ and $\lim_{i\rightarrow+\infty}\epsilon_i=0$. We have
\begin{align*}
    a_i&=a(E_i,X_i,B_i)\ge a(E_i,X_i,\bar B_i)\ge a(E_i,X_i,(1+\epsilon_i)B_i)\\
    &=a(E_i,X_i,B_i)-\epsilon_i\mult_{E_i}B_i>a_i-\epsilon_iM,
\end{align*}
Since $\lim_{i\rightarrow+\infty}a_i=\lim_{i\rightarrow+\infty}(a_i-\epsilon_iM)=a$, possibly passing to a subsequence, we may assume that $\bar{a_i}:=a(E_i,X_i,\bar B_i)$ is strictly increasing and $\lim_{i\rightarrow+\infty}\bar{a_i}=a$.

Let $f_i: Y_i\rightarrow X_i$ be the divisorial contraction which extracts $E_i$, and let $B_{Y_i},\bar B_{Y_i}$ be the strict transforms of $B_i$ and $\bar B_i$ on $Y_i$ respectively. Then $(Y_i/X_i\ni x_i,B_{Y_i}+(1-a)E_i)$ is canonical and $$\coeff(B_{Y_i}+(1-a)E_i)\subseteq\Ii\cup\{1-a\}.$$ By Theorem \ref{thm: alct acc terminal threefold}, possibly passing to a subsequence, we may assume that $(Y_i/X_i\ni x_i,\bar B_{Y_i}+(1-a)E_i)$ is canonical. 

By Theorem \ref{thm: ni decomposable complement}, there exists a positive integer $N$ and a finite set $\Ii_0\subset (0,1]$, such that $(X_i\ni x_i,\bar B_i)$ has an $(N,\Ii_0)$-decomposable $\Rr$-complement $(X_i\ni x_i,\bar B_i^+)$ for each $i$. In particular, $a(E_i,X_i,\bar B_i^+)$ belongs to a discrete, hence finite set for any $i$. Thus there exists a positive real number $t$, such that $(Y_i/X_i\ni x_i,\bar B_{Y_i}+(1-a+t)E_i)$ is lc. Possibly passing to a subsequence, we may assume that $a-\bar{a_i}<\frac{t}{2}$ for any $i$.

For any $i$, we let
$$\Dd_i:=\{F_i\mid F_i\text{ is over }X_i\ni x_i, F_i\not=E_i, a(F_i,X_i,\bar B_i)<1\}.$$
For any $F_i\in\Dd_i$, we have
$$a(F_i,Y_i,\bar B_{Y_i}+(1-a+t)E_i)\ge 0\text{ and }a(F_i,Y_i,\bar B_{Y_i}+(1-\bar{a_i})E_i)<1.$$
Since $a-\bar{a_i}<\frac{t}{2}$, $\mult_{F_i}E_i<\frac{2}{t}$, and 
\begin{align*}
a(F_i,X_i,\bar B_i)=&a(F_i,Y_i,\bar B_{Y_i}+(1-\bar{a_i})E_i)\\
=&a(F_i,Y_i,\bar B_{Y_i}+(1-a)E_i)+(a-a_i)\mult_{F_i}E_i>1-\frac{2(a-\bar{a_i})}{t}.
\end{align*}
Possibly passing to a subsequence, we may assume that $a-\bar{a_i}<\frac{t}{2}(1-a)$ for every $i$. Then $a(F_i,X_i,\bar B_i)>a>\bar{a_i}=a(E_i,X_i,\bar B_i)$ for any $F_i\in\Dd_i$.

If $(X_i,\bar B_i)$ is not klt near $x_i$ for infinitely many $i$, then we let $\phi_i: W_i\rightarrow X_i$ be a dlt modification of $(X_i,\bar B_i)$, and let $K_{W_i}+\bar B_{W_i}:=\phi_i^*(K_{X_i}+B_i)$. Then there exists a prime divisor $H_i\subset\Exc(\phi_i)$ such that $\Center_{W_i}E_i\subset H_i$. We immediately get a contradiction by applying adjunction to $H_i$ and using the precise inversion of adjunction formula (cf. \cite[Lemma 3.3]{Liu18}) and the ACC for mlds of surfaces \cite[Theorem 3.8]{Ale93}. Therefore, possibly passing to a subsequence, we may assume that $(X_i,\bar B_i)$ is klt near $x_i$ for each $i$.

Since $(X_i,\bar B_i)$ is klt near $x_i$, each $\Dd_i$ is a finite set, and we may assume that $\Dd_i=\{F_{i,1},\dots,F_{i,r_i}\}$ for some non-negative integer $r_i$. Note that $\lim_{i\to +\infty} a(F_{i,j},X_i,\bar B_i)=1$ for any $1\le j\le r_i$, and $\lim_{i\to +\infty} a(E_i,X_i,\bar B_i)=a$. By \cite[Lemma 5.3]{Liu18}, possibly reordering $F_{i,1},\dots,F_{i,r_i}$, one of the following hold:
\begin{enumerate}
    \item There exists a birational morphism $g_i: Z_i\rightarrow X_i$ which extracts exactly $F_{i,1},\dots,F_{i,r_i}$, such that
    $$\bar{a_i}=a(E_i,Z_i,\bar B_{Z_i})+\sum_{j=1}^{r_i}(1-a(F_{i,j},X_i,\bar B_i)F_{i,j}) \le a(E_i,Z_i,\bar B_{Z_i})<a$$
    for each $i$, where $\bar B_{Z_i}$ is the strict transform of $\bar B_i$ on $Z_i$. In this case, we let $E_i':=E_i$ and $B_i':=\bar B_{Z_i}$
    \item There exists a birational morphism $g_i: Z_i\rightarrow X_i$ which extracts exactly $E_i,F_{i,1},\dots,F_{i,r_i-1}$, such that
    $$1-\frac{2(a-\bar{a_i})}{t}\le a(F_{i,r_i},X_i,\bar B_i)\le a(F_{i,r_i},Z_i,\bar B_{Z_i}+(1-a)E_{Z_i})<1$$
    for each $i$, where $\bar B_{Z_i},E_{Z_i}$ are the strict transforms of $\bar B_i$ and $E_i$ on $Z_i$ respectively. In this case, we let $E_i'=F_{i,r_i}$ and $B_i':=\bar B_{Z_i}+(1-a)E_{Z_i}$.
\end{enumerate}
In either case, possibly passing to a subsequence, we may assume that $\bar{a_i}':=a(E_i',Z_i,B_i')<1$ is strictly increasing.
For any $i$, we let
$$\Dd'_i:=\{F_i'\mid F_i'\text{ is over }X_i\ni x_i, F_i'\text{ is exceptional over }Z_i, a(F_i',Z_i,B_i')=1\}.$$
By construction, $(Z_i/X_i\ni x_i,B_i')$ is klt, so $\Dd'_i$ is a finite set. Let $h_i: V_i\rightarrow W_i$ be a birational morphism which extracts all divisors in $\Dd'_i$, and let $B_{V_i}'$ be the strict transform of $B_i'$ on $V_i$. Then $(V_i,B_{V_i}')$ is enc and $a_i':=a(E_i',V_i,B_{V_i'})$ is strictly increasing. Moreover, $\lim_{i\rightarrow+\infty}a_i'=1$ or $a$. Since $a>\epsilon$, possibly passing to a subsequence, we may assume that $a_i'>\epsilon$ for any $i$. However, the coefficients of $B_{V_i}'$ belong to a finite set, which contradicts our assumptions.
\end{proof}

\section{Index one cover}
\subsection{Enc cyclic quotient singularities}
In this subsection, we prove Theorem \ref{thm: acc mld enc dim 3} when $X$ is non-canonical, with isolated singularities, and the index $1$ cover of $X$ is smooth (see Theorem \ref{thm: enc smooth cover case}). 

\begin{lem}[cf. {\cite[Lemma 2.11]{LL22}, \cite[Theorem 1]{Amb06}}]\label{lem: set of cyc lds}
Let $d$ be a positive integer and $(X\ni x)=\frac{1}{r}(a_1,a_2,\dots,a_d)$ a $d$-dimensional cyclic quotient singularity. Let $$\bm{e}:=\left(\left\{\frac{a_1}{r}\right\},\left\{\frac{a_2}{r}\right\},\dots,\left\{\frac{a_d}{r}\right\}\right),$$ 
$$\bm{e}_i \text{ the }i\text{-th unit vector in } \mathbb Z^d\text{ for any }1\le i\le d,$$ $$N:=\mathbb Z_{\ge 0}\bm{e}\oplus\mathbb Z_{\ge 0}\bm{e}_1\oplus\mathbb Z_{\ge 0}\bm{e}_2\oplus\dots\oplus\mathbb Z_{\ge 0}\bm{e}_d,$$ 
$$\sigma:=N\cap\mathbb Q_{\ge 0}^d, \text{ and }\relin(\sigma):=N\cap\mathbb Q_{>0}^d.$$ The following holds.
\begin{enumerate}
\item For any prime divisor $E$ over $X\ni x$ that is invariant under the cyclic quotient action, there exists a primitive vector $\alpha\in\relin(\sigma)$ such that $a(E,X,0)=\alpha(x_1x_2\cdots x_d)$. In particular, there exists a unique positive integer $k\le r$, such that $$\alpha\in\left(1+\frac{a_1k}{r}-\left\lceil\frac{a_1k}{r}\right\rceil,1+\frac{a_2k}{r}-\left\lceil\frac{a_2k}{r}\right\rceil,\dots,1+\frac{a_dk}{r}-\left\lceil\frac{a_dk}{r}\right\rceil\right)+\mathbb Z^d_{\ge 0}.$$
    \item $$\mld(X\ni x)=\min_{1\le k\le r-1}\left\{\sum_{i=1}^d\left(1+\frac{ka_i}{r}-\left\lceil\frac{ka_i}{r}\right\rceil\right)\right\}\le d.$$
\end{enumerate}
\end{lem}
\begin{proof}
Point (1) is elementary toric geometry, and (2) follows immediately from (1). 
\end{proof}

\begin{lem}\label{lem: toric bdd index lemma}
Let $d$ be a positive integer and $\epsilon$ a positive real number. Then there exists a positive integer $I$, depending only on $d$ and $\epsilon$, satisfying the following. Let $r$ be a positive integer and $v_1,\dots,v_d\in [0,1]$ real numbers, such that  $\sum_{i=1}^d(1+(m-1)v_i-\lceil mv_i\rceil)\ge\epsilon$ for any $m\in [2,r]\cap \Zz$. Then $r\le I$.
\end{lem}
\begin{proof}
Suppose that the statement does not hold. Then for each $j\in\Zz_{\ge 1}$, there exist $v_{1,j},\dots,v_{d,j}\in [0,1]$ and positive integers $r_j$, such that
\begin{itemize}
    \item $\sum_{i=1}^d(1+(m-1)v_{i,j}-\lceil mv_{i,j}\rceil)\ge\epsilon$
for any $m\in [2,r_j]\cap\Zz$,
\item $r_j$ is strictly increasing, and
\item $\bar v_i:=\lim_{j\rightarrow+\infty}v_{i,j}$ exists.
\end{itemize}
Let $\bm{v}:=(\bar v_1,\dots,\bar v_d)$. 
By Kronecker's theorem, there exist a positive integer $n$ and a vector $\bm{u}\in\mathbb Z^{d}$ such that $||n\bm{v}-\bm{u}||_{\infty}<\min\{\frac{\epsilon}{d},\bar v_i\mid \bar v_i>0\}$ and $n\bar v_i\in\Zz$ for any $i$ such that $\bar v_i\in\mathbb Q$. In particular, $\lceil (n+1)\bar v_i\rceil=\lfloor (n+1)\bar v_i\rfloor+1$ for any $i$ such that $\bar v_i\in (0,1)$. Now $\lim_{j\rightarrow+\infty}(1+nv_{i,j}-\lceil (n+1)v_{i,j}\rceil)=0$ when $\bar v_i=0$ and $\lim_{j\rightarrow+\infty}(1+nv_{i,j}-\lceil (n+1)v_{i,j}\rceil)=1+n\bar v_i-\lceil (n+1)\bar v_i\rceil$ when $\bar v_i>0$. Thus
\begin{align*}
    &\lim_{j\rightarrow+\infty}\sum_{i=1}^d(1+nv_{i,j}-\lceil (n+1)v_{i,j}\rceil)=\sum_{0<\bar v_i<1}(1+n\bar v_i-\lceil (n+1)\bar v_i\rceil)\\
    =&\sum_{0<\bar v_i<1}(1+(n+1)\bar{v_i}-\lceil (n+1)\bar v_i\rceil-\bar v_i)=\sum_{\bar 0<\bar v_i<1}(\{(n+1)\bar v_i\}-\bar v_i)<\sum_{0<\bar v_i<1}\frac{\epsilon}{d}\le \epsilon.
\end{align*}
Thus possibly passing to a subsequence, $\sum_{i=1}^d(1+nv_{i,j}-\lceil (n+1)v_{i,j}\rceil)<\epsilon$ for any $j$, hence $n>r_j$, which contradicts $\lim_{j\rightarrow+\infty}r_j=+\infty$.
\end{proof}

\begin{thm}\label{thm: enc smooth cover case}
Let $\Ii\subset [0,1]$ be a DCC (resp. finite) set. Assume that $(X\ni x,B)$ is a $\Qq$-factorial enc pair of dimension $3$, such that
\begin{enumerate}
    \item $X\ni x$ is a non-canonical isolated singularity,
    \item $\coeff(B)\subseteq\Ii$, and
    \item $\tilde X\ni\tilde x$ is smooth, where $\pi: (\tilde X\ni\tilde x)\rightarrow (X\ni x)$ is the index $1$ cover of $X\ni x$.
\end{enumerate}
Then $\mld(X\ni x, B)$ belongs to an ACC set (resp. is discrete away from $0$).
\end{thm}
\begin{proof}
We may assume that $\mld(X\ni x,B)>\epsilon$ for some fixed positive real number $\epsilon<\frac{12}{13}$. Then $\mld(X\ni x)>\epsilon$. Since $X\ni x$ is non-canonical and $(X\ni x,B)$ is enc, $X\ni x$ is enc. Since $X\ni x$ is an isolated singularity, by Theorem \ref{thm: 12/13}, $\mld(X\ni x)=\mld(X)\le \frac{12}{13}$. Since $\tilde X$ is smooth, $(X\ni x)$ is analytically isomorphic to a cyclic quotient singularity $(Y\ni y)=\frac{1}{r}(a_1,a_2,a_3)$, where $a_1,a_2,a_3,r$ are positive integers such that $a_i<r$ and $\gcd(a_i,r)=1$ for each $i$. By Lemma \ref{lem: set of cyc lds}, there exists a positive integer $k_0\in [1,r-1]$ such that $\mld(Y\ni y)=\mld(X\ni x)=\sum_{i=1}^3\{\frac{a_ik_0}{r}\}\in [\epsilon,\frac{12}{13}]$, and for any positive integer $k\not=k_0$ such that $k\in [1,r-1]$, $\sum_{i=1}^3\{\frac{a_ik}{r}\}>\min\{1, 2\,\mld(X\ni x)\}$. Let $v_i:=\{\frac{a_ik_0}{r}\}$ for each $i$. Then
$$\sum_{i=1}^3(1+(m-1)v_i-\lceil mv_i\rceil)=\sum_{i=1}^3(1+mv_i-\lceil mv_i\rceil)-\sum_{i=1}^3 v_i>\min\{\frac{1}{13},\epsilon\}$$
    for any $m\in [2,\frac{r}{\gcd(k_0,r)}-1]\cap\Zz_{\ge 1}$.
By Lemma \ref{lem: toric bdd index lemma}, $\frac{r}{\gcd(k_0,r)}$ belongs to a finite set. In particular, $w:=(v_1,v_2,v_3)$ belongs to a finite set, hence $\mld(Y\ni y)$ belongs to a finite set. Suppose that $B=\sum_i b_iB_i$ where $B_i$ are the irreducible components of $B$ and $b_i\in\Ii_{>0}$, and $B_i\cong (f_i=0)|_Y$ for some semi-invariant analytic function $f_i\in\mathbb C\{x_1,x_2,x_3\}$. Let $E$ be the unique divisor over $X\ni x$ such that $a(E,X,0)\le 1$. Then by Lemma \ref{lem: weighted blowup log discrepancies}, 
\begin{align*}
    \mld(X\ni x,B)=&a(E,X,B)=a(E,X,0)-\mult_E B\\
    =&a(E,X,0)-\sum_{i} b_i\mult_EB_i=\mld(Y\ni y)-\sum_{i} b_iw(f_i)
\end{align*}
belongs to an ACC set (resp. finite set).
\end{proof}

\subsection{Enc cDV quotient singularities} In this subsection, we prove Theorem \ref{thm: acc mld enc dim 3} when $X$ is non-canonical, with isolated singularities, and the index $1$ cover of $X$ is cDV (see Theorem \ref{thm: enc cDV cover case}).

\begin{thm}\label{thm: jia21 2.4 enc case}
Let $r$ be a positive integer, $a,b,c,d$ integers, and $f\in\mathbb C\{x_1,x_2,x_3,x_4\}$ an analytic function, such that $$(X\ni x)\cong\left((f=0)\subset(\mathbb C^4\ni 0)\right)/\frac{1}{r}(a,b,c,d)$$
is an enc threefold isolated singularity. Let $$N:=\{w\in\mathbb Q^4_{\ge 0}\mid w\equiv\frac{1}{r}(ja,jb,jc,jd)\mod \mathbb Z^4\text{ for some }j\in\mathbb Z\}\backslash\{\bm{0}\}.$$
Then there exists at most one primitive vector $\beta\in N$, such that $t:=\beta(x_1x_2x_3x_4)-\beta(f)\le 1$. 

In particular, for any $\alpha\in N\backslash\{\beta,2\beta,\dots,(k-1)\beta\}$, $\alpha(x_1x_2x_3x_4)-\alpha(f)>1$, where $k:=\lfloor \frac{1}{t}\rfloor +1$.
\end{thm}

\begin{proof}
Assume that there exists a primitive vector $\beta\in N$ such that $\beta(x_1x_2x_3x_4)-\beta(f)\le 1$. It suffices to show that such $\beta$ is unique.

Let $Z:=\Cc^4/\frac{1}{r}(a,b,c,d)$, and $\phi_{\beta}: Z_{\beta}\to Z$ the toric morphism induced by $\beta$ which extracts an exceptional divisor $E_{\beta}$. Let $X_{\beta}$ be the strict transform of $X$ on $Z_{\beta}$. By \cite[Proposition 2.1]{Jia21}, we have 
$$K_{Z_{\beta}}+X_{\beta}+(1-t)E_{\beta}=\phi_{\beta}^{*}(K_Z+X).$$
Since $X\ni x$ is an isolated klt singularity and $\dim X=3$, $(Z,X)$ is plt by inversion of adjunction. Thus $({Z_{\beta}},X_{\beta}+(1-t)E_{\beta})$ is plt. By the adjunction formula, $$K_{X_{\beta}}+B_{\beta}:=(K_{Z_{\beta}}+X_{\beta}+(1-t)E_{\beta})|_{X_{\beta}}=\phi_{\beta}^{*}K_X,$$
for some $B_{\beta}\ge0$, and the coefficients of $B_{\beta}$ are of the form $1-\frac{1-s(1-t)}{l}$ for some positive integers $l,s$ as $E_{\beta}$ intersects $X_{\beta}$. Since $X$ is enc, $\Supp B_{\beta}$ is a prime divisor, say $F_{\beta}$.

Let $v_{F_{\beta}}$ be the divisorial valuation of $F_{\beta}$. Thus $v_{F_{\beta}}(\bm{x}^{\bm{m}})=(1-\frac{1-s(1-t)}{l})\beta(\bm{x}^{\bm{m}})$ for any monomial $\bm{x}^{\bm{m}}$, where $\bm{m}\in M$, and $M$ is the dual sublattice of $\Zz^4+\Zz\cdot\frac{1}{r}(a,b,c,d)$. Hence such $\beta$ is unique by the primitivity.
\end{proof}

We introduce the following setting. Roughly speaking, Theorem \ref{thm: jia21 rules enc case} below will show that if $$(X\ni x)\cong (f=0)\subset (\mathbb C^4\ni 0)/\frac{1}{r}(a_1,a_2,a_3,a_4)$$
is enc and a cyclic quotient of an isolated cDV singularity, then $f,a_i,r,e,k$, and $\beta$ should satisfy Setting~\ref{Setting: before terminal lem}. Therefore, we can transform the ACC conjecture in this case to computations on variables that satisfy Setting~\ref{Setting: before terminal lem}. 

\begin{sett}\label{Setting: before terminal lem}
We set up the following notation and conditions.
\begin{enumerate}
\item Let $r$ be a positive integer, $0\le a_1,a_2,a_3,a_4,e<r$ integers, such that \begin{enumerate}
       \item $\gcd(a_i,r)\mid\gcd(e,r)$ for any $1\le i\le 4$.
    \item $\gcd(a_i,a_j,r)=1$ for any $1\le i<j\le 4$.
    \item  $\sum_{i=1}^4a_i-e\equiv 1\mod r$.
\end{enumerate}
\item $f\in \mathbb C\{x_1,x_2,x_3,x_4\}$ is $\bm{\mu}$-semi-invariant, that is, ${\bm{\mu}}(f)=\xi^ef$, and is one of the following $3$ types:
\begin{enumerate}
    \item (cA type) $f=x_1x_2+g(x_3,x_4)$ with $g\in\mm^2$.
    \item (Odd type) $f=x_1^2+x_2^2+g(x_3,x_4)$ with $g\in\mm^3$ and $a_1\not\equiv a_2\mod r$.
    \item (cD-E type) $f=x_1^2+g(x_2,x_3,x_4)$ with $g\in\mm^3$,
\end{enumerate}
where $\mm$ is the maximal ideal of $\mathbb C\{x_1,x_2,x_3,x_4\}$, and ${\bm{\mu}}:\mathbb C^4\rightarrow\mathbb C^4$ is the action $(x_1,x_2,x_3,x_4)\rightarrow (\xi^{a_1}x_1,\xi^{a_2}x_2,\xi^{a_3}x_3,\xi^{a_4}x_4)$.
\item 
One of the two cases hold:
\begin{enumerate}
    \item $\alpha(x_1x_2x_3x_4)-\alpha(f)>1$ for any $\alpha\in N$. In this case, we let $k:=1$ and $\beta:=\bm{0}$.
    \item There exists an integer $k\ge 2$, and a primitive vector $\beta\in N$, such that
    \begin{enumerate}
        \item 
        \begin{itemize}
            \item either $\frac{1}{k}<\beta(x_1x_2x_3x_4)-\beta(f)\le \min\{\frac{12}{13},\frac{1}{k-1}\}$, or
            \item $\beta(x_1x_2x_3x_4)-\beta(f)=1$ and $k=2$, 
        \end{itemize} 
        and
        \item for any $\alpha\in N\backslash\{\beta,2\beta,\dots,(k-1)\beta\}$, $\alpha(x_1x_2x_3x_4)-\alpha(f)>1$,
    \end{enumerate}
\end{enumerate}
where $$N:=\{w\in\mathbb Q^4_{\ge 0}\mid w\equiv\frac{1}{r}(ja_1,ja_2,ja_3,ja_4)\mod \mathbb Z^4\text{ for some }j\in\mathbb Z\}\backslash\{\bm{0}\}.$$
\end{enumerate}
Moreover, if $f$ is of cA type, then for any integer $a$ such that $\gcd(a,r)=1$, $\frac{1}{r}(a_1,a_2,a_3,a_4,e)\not\equiv\frac{1}{r}(a,-a,1,0,0)\mod \Zz^5$. 
\end{sett}

\begin{thm}\label{thm: jia21 rules enc case}
Let $r$ be a positive integer, $0\le a_1,a_2,a_3,a_4,e<r$ integers, $\xi:=e^{\frac{2\pi i}{r}}$, $$N:=\{w\in\mathbb Q^4_{\ge 0}\mid w\equiv\frac{1}{r}(ja_1,ja_2,ja_3,ja_4)\mod \mathbb Z^4\text{ for some }j\in\mathbb Z\}\backslash\{\bm{0}\},$$
${\bm{\mu}}:\mathbb C^4\rightarrow\mathbb C^4$ the action  $(x_1,x_2,x_3,x_4)\rightarrow (\xi^{a_1}x_1,\xi^{a_2}x_2,\xi^{a_3}x_3,\xi^{a_4}x_4)$, and $f\in\mathbb C\{x_1,x_2,x_3,x_4\}$ a $\bm{\mu}$-semi-invariant analytic function such that ${\bm{\mu}}(f)=\xi^ef$. Suppose that $$(X\ni x)\cong (f=0)\subset (\mathbb C^4\ni 0)/\frac{1}{r}(a_1,a_2,a_3,a_4)$$
be a hyperquotient singularity such that  
\begin{itemize}
    \item $(Y\ni y):=(f=0)\cong (\mathbb C^4\ni 0)$ is an isolated cDV singularity,
    \item $\pi: (Y\ni y)\rightarrow (X\ni x)$ is the index one cover, and
    \item $(X\ni x)$ is enc, 
\end{itemize}
then possibly replacing $\frac{1}{r}(a_1,a_2,a_3,a_4)$ with $(\{\frac{ja_1}{r}\},\{\frac{ja_2}{r}\},\{\frac{ja_3}{r}\},\{\frac{ja_4}{r}\})$ for some $j$ such that $\gcd(j,r)=1$, and taking a $\bm{\mu}$-equivariant analytic change of coordinates and possibly permuting the coordinates $x_i$, we have that $a_i,e,r,f$ satisfy Setting~\ref{Setting: before terminal lem}.
\end{thm}

\begin{proof}
By \cite[Page 394]{Rei87}, since $\bm{\mu}$ acts freely outside $y$, $a_i,e,r$ satisfy Setting~\ref{Setting: before terminal lem}(1.a) and (1.b).
Let $s\in\omega_Y$ be a generator, then $\bm{\mu}$ acts on $s$ by $s\rightarrow \xi^{\sum_{i=1}^4a_i-e}s$. Since the Cartier index of $K_X$ near $x$ is $r$, $\gcd(\sum_{i=1}^4a_i-e,r)=1$, and $a_i,e,r$ satisfy Setting~\ref{Setting: before terminal lem}(1.c). By \cite[Page 394-395]{Rei87} and \cite[(6.7) Proposition]{Rei87} (see also \cite[Proposition 4.2]{Jia21}), $f$ satisfies Setting~\ref{Setting: before terminal lem}(2).

By Theorem \ref{thm: jia21 2.4 enc case}, in order to show that $f$ satisfies Setting~\ref{Setting: before terminal lem}(3), we only need to prove that $$\beta(x_1x_2x_3x_4)-\beta(f)\not\in (\frac{12}{13},1).$$ We may assume that $r>13$. By \cite[Theorem 1.6, Lemmas 6.3 and 6.4]{LX21} and \cite[Lemma 2.12, Remark 2.13]{Jia21}, if $\beta(x_1x_2x_3x_4)-\beta(f)\in (\frac{12}{13},1)$, then $(X\ni x)$ and $\beta$ satisfy \cite[Section 4, Rules I-III]{Jia21}, which is absurd according to \cite[Section 4]{Jia21}. 

It suffices to show that $a_i,e,r$ satisfy the ``moreover'' part of Setting~\ref{Setting: before terminal lem}. By \cite[Theorem 6.5]{KSB88}, if $\frac{1}{r}(a_1,a_2,a_3,a_4,e)\equiv\frac{1}{r}(a,-a,1,0,0)\mod \Zz^5$, then $X\ni x$ is a terminal singularity, which leads to a contradiction.
\end{proof}

\begin{thm}\label{thm: beta finite}
With notation and conditions as in Setting \ref{Setting: before terminal lem}, either $r$ or $\beta\not=\bm{0}$ belongs to a finite set depending only on $k$. In particular, $\beta(x_1x_2x_3x_4)-\beta(f)$ belongs to a finite set depending only on $k$, and $\beta(g)$ belongs to a discrete set for any analytic function $g$.
\end{thm}

\begin{proof}
Proving that either $r$ or $\beta\not=\bm{0}$ belongs to a finite set depending only on $k$ is elementary but requires complicated computations, so we omit the proof and refer the reader to \cite[Theorem 1.2]{HL25} (which was Theorem A.1 of the first arXiv version\footnote{See \href{https://arxiv.org/abs/2209.13122v1}{ arXiv:2209.13122v1}.} of the present paper). 

We are left to prove the ``in particular''-part. There exists a positive integer $n$ depending only on $k$ such that $n\beta\in \mathbb Z_{\ge 0}^4$. Thus $\beta(g)$ belongs to the discrete set $\frac{1}{n}\mathbb Z_{\ge 0}$ for any analytic function $g$. Since $\beta(x_1x_2x_3x_4)-\beta(f)\in (0,1]$, $\beta(x_1x_2x_3x_4)-\beta(f)$ belongs to the finite set $\frac{1}{n}\mathbb Z_{\ge 0}\cap (0,1]$. 
\end{proof}

\begin{thm}\label{thm: enc cDV cover case}
Let $\Ii\subset[0,1]$ be a DCC (resp. finite) set. Assume that $(X\ni x,B)$ is a $\Qq$-factorial enc pair of dimension $3$, such that
\begin{enumerate}
    \item $X\ni x$ is an isolated non-canonical singularity, 
    \item $\coeff(B)\subseteq \Ii$, and
    \item $\tilde X\ni\tilde x$ is terminal but not smooth, where $\pi: (\tilde X\ni\tilde x)\rightarrow (X\ni x)$ is the index $1$ cover of $X\ni x$.
\end{enumerate}
Then $\mld(X\ni x,B)$ belongs to an ACC set (resp. is discrete away from $0$).
\end{thm}
\begin{proof}
We only need to show that for any positive integer $l\ge 2$, if $\mld(X\ni x)\in (\frac{1}{l},\frac{1}{l-1}]$, then $\mld(X\ni x,B)$ belongs to an ACC set (resp. is discrete away from $0$). 

There exists a positive integer $r$, integers $0\le a_1,a_2,a_3,a_4,e$, and $\xi:=e^{\frac{2\pi i}{r}}$, such that $$(X\ni x)\cong \left((f=0)\subset(\mathbb C^4\ni 0)\right)/\bm{\mu},$$ where $\bm{\mu}:\mathbb C^4\rightarrow\mathbb C^4$ is the action $(x_1,x_2,x_3,x_4)\rightarrow (\xi^{a_1}x_1,\xi^{a_2}x_2,\xi^{a_3}x_3,\xi^{a_4}x_4)$
and $f$ is $\bm{\mu}$-semi-invariant, such that $\bm{\mu}(f)=\xi^ef$. By Setting~\ref{Setting: before terminal lem}(1.c), possibly replacing $(a_1,a_2,a_3,a_4)$ and $e$, we may assume that $a_1+a_2+a_3+a_4-e\equiv 1\mod r$. Moreover, possibly shrinking $X$ to a neighborhood of $x$, we may write $B=\sum_{i=1}^mb_iB_i$ where $B_i$ are the irreducible components of $B$ and $x\in\Supp B_i$ for each $i$. Then $b_i\in\Ii$, and we may identify $B_i$ with $\left((f_i=0)\subset(\mathbb C^4\ni 0)\right)/\bm{\mu}|_{X}$ for some $\bm{\mu}$-semi-invariant function $f_i$ for each $i$.  

Let $$N:=\{w\in\mathbb Q_{\ge 0}^4\mid w=\frac{1}{r}(ja_1,ja_2,ja_3,ja_4)\mod \Zz^4\text{ for some }j\in\mathbb Z\}\backslash\{\bm{0}\}.$$

By Setting~\ref{Setting: before terminal lem}(3), there are two cases:

\medskip

\noindent\textbf{Case 1}. $\alpha(x_1x_2x_3x_4)-\alpha(f)>1$ for any $\alpha\in N$. In this case, by Theorems \ref{thm: jia21 rules enc case}, \ref{thm: beta finite}, $r$ belongs to a finite set. Since $(X\ni x,B)$ is enc and $X\ni x$ is non-canonical, there exists a unique prime divisor $E$ over $X\ni x$, such that $a(E,X,B)=\mld(X\ni x,B)$ and $a(E,X,0)<1$. Since $rK_X$ is Cartier, $r\,a(E,X,0)$ belongs to a finite set and $r\mult_EB_i\in\Zz_{\ge 1}$ for each $i$. Thus $$a(E,X,B)=a(E,X,0)-\sum_{i=1}^mb_i\mult_EB_i$$
belongs to an ACC set (resp. finite set).

\medskip

\noindent\textbf{Case 2}. There exists a unique primitive vector $\beta\in N$ and an integer $k\ge 2$, such that $$\frac{1}{k}<\beta(x_1x_2x_3x_4)-\beta(f)\le \frac{1}{k-1}.$$ We consider the pair $$(Z\ni z,X+B_Z):=\left(\mathbb C^4\ni 0, (f=0)+\sum_{i=1}^mb_i(f_i=0)\right)/\bm{\mu}.$$
By \cite[Proposition 2.1]{Jia21}, the primitive vector $\beta\in N$ corresponds to a divisor $E$ over $Z\ni z$, and 
\begin{equation}\label{eqn: a beta}
 a:=a(E,Z,X+B_Z)=\beta(x_1x_2x_3x_4)-\beta(f)-\sum_{i=1}^mb_i\beta(f_i=0).
\end{equation}
In particular, $0<a\le \frac{1}{k-1}\le 1$. Let $h: W\rightarrow Z$ be the birational morphism which extracts $E$, then
$$K_W+X_W+B_W+(1-a)E=h^*(K_Z+X+B_Z),$$
where $X_W$ and $B_W$ are the strict transforms of $X$ and $B$ on $W$ respectively. 

Since $(X,B)$ is klt near $x$, $X\ni x$ is an isolated singularity, and $\dim X=3$, $(Z,X+B_Z)$ is plt by the inversion of adjunction. Thus $(W,X_W+B_W+(1-a)E)$ is plt. Since $h$ is a divisorial contraction of $E$ and $\Center_Z E=x$, $E$ is $\Qq$-Cartier, and $\Supp (E\cap X_W)$ contains a prime divisor $F$ which does not belong to $\Supp (B_W\cap E)$. By the adjunction formula, 
$$a(F,X,B)=\frac{1}{n}(1-s(1-a))\le 
a\le \frac{1}{k-1}$$
for some positive integers $n,s$. Since $(X\ni x,B)$ is enc, $a(F,X,B)=\mld(X\ni x,B)>\frac{1}{l}$, hence $k\le l$. Thus $k$ belongs to a finite set. By Theorems \ref{thm: jia21 rules enc case} and \ref{thm: beta finite}, $a\le 1$ belongs to an ACC set (resp. finite set). Thus 
$$\mld(X\ni x,B)=a(F,X,B)=\frac{1}{n}(1-s(1-a))$$ belongs to an ACC set (resp. is discrete away from $0$), and we are done.
\end{proof}

\section{Proofs of other main results}

\begin{lem}\label{lem: e to n}
For any positive integer $l$, Theorem \hyperlink{thm: E}{E}$_l$ implies Theorem \hyperlink{thm: N}{N}$_l$.
\end{lem}
\begin{proof}
This follows from Lemma \ref{lem: reduce to enc}.
\end{proof}

\begin{lem}\label{lem: c to e}
For any positive integer $l$, Theorem \hyperlink{thm: C}{C}$_l$ implies Theorem \hyperlink{thm: E}{E}$_l$.
\end{lem}

\begin{proof}   
Let $(X,B)\in\mathcal{E}(l,\Ii)$, and $E$ the unique exceptional prime divisor over $(X,B)$ such that $a(E,X,B)=\mld(X,B)$. 

By Lemma \ref{lem: acc enc center curve case}, we may assume that $x:=\Center_XE$ is a closed point. By Theorem \ref{thm: enc terminal reduce to fin coeff}, we may assume that either $X$ is terminal and $\Ii$ is a finite set, or $X$ is not terminal. By Theorem \ref{thm: enc ACC fin coeff}, we may assume that $X$ is not terminal. By Theorem \ref{thm: acc mld enc dim 3 strictly canonical case}, we may assume that $X$ is not canonical. Let $(\tilde X\ni\tilde x)\rightarrow (X\ni x)$ be the index $1$ cover of $X\ni x$.  Then $\tilde X\ni\tilde x$ is smooth, or an isolated cDV singularity, or a strictly canonical singularity. By Theorems \ref{thm: enc smooth cover case} and \ref{thm: enc cDV cover case}, we may assume that $\tilde X\ni\tilde x$ is strictly canonical. By Theorem \hyperlink{thm: C}{C}$_l$, $\mld(X,B)$ belongs to an ACC set.
\end{proof}

\begin{lem}\label{lem: n to c}
For any positive integer $l\ge 2$, Theorem \hyperlink{thm: N}{N}$_{l-1}$ implies Theorem \hyperlink{thm: C}{C}$_l$.
\end{lem}
\begin{proof}
Let $(X\ni x,B)\in\mathcal{C}(l,\Ii)$, $\pi: (\tilde X\ni\tilde x)\rightarrow (X\ni x)$ the index $1$ cover of $X\ni x$, and $\tilde B:=\pi^*B$. then $\coeff(\tilde B)\subseteq\Ii$. 

Since $\mld(X)<1$, there exists an exceptional prime divisor $E$ over $X$, such that $a(E,X,B)\le a(E,X,0)=\mld(X)< 1$. Thus $E$ is the unique exceptional prime divisor over $X$ such that $a(E,X,B)\le 1$ as $(X\ni x,B)$ is enc. In particular, $a(E,X,B)=\mld(X\ni x,B)$. Hence for any exceptional prime divisor $\tilde E$ over $\tilde X$ such that $a(\tilde E,\tilde X,\tilde B)\le 1$, we have $a(\tilde E,\tilde X,\tilde B)=r_{\tilde E}a(E,X,B)$, where $r_{\tilde E}$ is the ramification index of $\pi$ along $\tilde E$. Since $(\tilde X\ni \tilde x)$ is canonical, 
$$1\le a(\tilde E,\tilde X,0)=r_{\tilde E}a(E,X,0)<r_{\tilde E},$$
so $r_{\tilde E}\ge 2$ for any $\tilde E$. It follows that
\begin{align*}
\mld(\tilde X,\tilde B)\in &\{a(\tilde E,\tilde X,\tilde B)\le 1\mid \tilde E\text{ is exceptional over }\tilde X\}
\subseteq\{2\,\mld(X,B),\dots,(l-1)\,\mld(X,B)\}\end{align*}
as $\mld(X,B)>\frac{1}{l}$ and $(\tilde X\ni \tilde x)$ is strictly canonical. In particular,
$$1\le \#(\{a(\tilde E,\tilde X,\tilde B)\mid \tilde E\text{ is exceptional over }\tilde X\}\cap [0,1])\le l-2.$$ 
Moreover, since $a(\tilde E,\tilde X,\tilde B)=r_{\tilde E}a(E,X,B)>\frac{2}{l}\ge\frac{1}{l-1}$, we have $\mld(\tilde X,\tilde B)>\frac{1}{l-1}$. Thus by Theorem \hyperlink{thm: N}{N}$_{l-1}$, $\mld(\tilde X,\tilde B)$ belongs to an ACC set, which implies that $\mld(X,B)$ also belongs to an ACC set.
\end{proof}

\begin{proof}[Proof of Theorems \hyperlink{thm: E}{E}, \hyperlink{thm: N}{N}, and \hyperlink{thm: C}{C}]
These follow from Lemmas \ref{lem: e to n}, \ref{lem: c to e}, \ref{lem: n to c}.
\end{proof}

\begin{proof}[Proof of Theorem \ref{thm: acc mld 3dim bounded ld number}]
By \cite[Theorem 1.1]{HLL22}, we may assume that $\mld (X,B)<1$. Now the theorem follows from Theorem \hyperlink{thm: N}{N}.
\end{proof}

\begin{proof}[Proof of Theorem \ref{thm: acc mld enc dim 3}] This follows from Theorem \ref{thm: acc mld 3dim bounded ld number}.
\end{proof}

\begin{cor}[=Corollary \ref{cor: tof dim 3}]\label{cor: new proof tof dim 3}
Any sequence of lc flips 
$$(X,B):=(X_0,B_0)\dashrightarrow (X_1,B_1)\dashrightarrow\dots (X_i,B_i)\dashrightarrow\dots$$
terminates in dimension $3$.
\end{cor}
\begin{proof}
We only need to check the conditions of Theorem \ref{thm: ter acc lsc+enc to tof} when $d=3$. Theorem \ref{thm: ter acc lsc+enc to tof}(1) follows from \cite[Corollary 1.5]{Nak16} (see also \cite[Theorem 1.1]{HLL22}), Theorem \ref{thm: ter acc lsc+enc to tof}(2) follows from \cite[Main Theorem 1]{Amb99} (see also \cite[Theorem 1.2]{NS21}), and Theorem \ref{thm: ter acc lsc+enc to tof}(3) follows from Theorem \ref{thm: acc mld enc dim 3}.
\end{proof}

We conjecture that Corollary \ref{cor: acc ld 3dim bounded ld number} generalizes to high dimensions:
\begin{conj}\label{conj: acc ld bounded ld number}
Let $N$ be a non-negative integer, $d$ a positive integer, and $\Ii\subset [0,1]$ a DCC set. Then there exists an ACC set $\Ii'$ depending only on $d,N$ and $\Ii$ satisfying the following. Assume that $(X,B)$ is a klt pair of dimension $d$, such that  
\begin{enumerate}
    \item $\coeff(B)\subseteq\Ii$, and
    \item there are at most $N$ different (exceptional) log discrepancies of $(X,B)$ that are $\le 1$, i.e.
    $$\#(\{a(E,X,B)\mid E\text{ is exceptional over }X\}\cap [0,1])\le N,$$
\end{enumerate}
then $\{a(E,X,B)\mid E\text{ is exceptional over }X\}\cap [0,1]\subset\Ii'$.
\end{conj}

\begin{thm}\label{thm: reduction to enc acc mld}
Assume that Conjecture \ref{conj: acc mld enc}(1) holds in dimension $d$. Then Conjecture \ref{conj: acc ld bounded ld number} holds in dimension $d$.
\end{thm}

\begin{proof}
This follows from Lemma \ref{lem: reduce to enc}.
\end{proof}

\begin{proof}[Proof of Corollary \ref{cor: acc ld 3dim bounded ld number}]
This follows from Theorems \ref{thm: acc mld enc dim 3} and \ref{thm: reduction to enc acc mld}.
\end{proof}

Finally, we show the following theorem for independent interest. Theorem \ref{thm: ecn 1 gap implies klt 1 gap} implies that in order to show the 1-gap conjecture for mlds (cf. \cite[Conjecture 5.4]{CDCHJS21}), it suffices to show the 1-gap conjecture for mlds of enc pairs. We note that the 1-gap conjecture has a close relation with the birational boundedness of rationally connected klt Calabi–Yau varieties, see \cite[Corollary 5.5]{CDCHJS21}, \cite{HJ24}.

\begin{thm}\label{thm: ecn 1 gap implies klt 1 gap}
Let $d$ be a positive integer, and $\Ii\subset[0,1]$ a set. Then
\begin{align*}
&\sup\left\{\mld (X,B)< 1\Bigm| (X,B)\text{ is }\Qq\text{-factorial enc,}\dim X=d, \coeff(B)\subseteq\Ii\right\}\\
=&\sup\left\{\mld (X,B)< 1\Bigm| (X,B)\text{ is klt},\dim X=d, \coeff(B)\subseteq\Ii\right\}.
\end{align*}
\end{thm}
\begin{proof}
Let $(X,B)$ be a klt pair such that $\dim X=d$, $\coeff(B)\subseteq \Ii$, and $\mld(X,B)<1$. By \cite[Proposition 2.36]{KM98}, we may assume that $E_1,E_2,\ldots,E_k$ are all exceptional prime divisors over $X$, such that $a(E_i,X,B)<1$.

By \cite[Lemma 5.3]{Liu18}, there exist $1\le i\le k$ and a birational morphism $f:Y\to X$ which extracts exactly all $E_1,E_2,\ldots,E_k$ but $E_i$, such that $1>a(E_i,Y,f_{*}^{-1}B)\ge a(E_i,X,B)$. Possibly replacing $Y$ with a small $\Qq$-factorialization, we may assume that $Y$ is $\Qq$-factorial. Then $(Y,f_{*}^{-1}B)$ is a $\Qq$-factorial enc pair with $1>\mld(Y,f_{*}^{-1}B)\ge \mld(X,B)$, and we are done.  
\end{proof}

\section{Further remarks}\label{sec: remark}

\begin{rem}[History of enc pairs]
We briefly introduce some history on the study of enc pairs. In \cite[Lemma 5.4]{Liu18}, a class of pairs similar to enc pairs, that is, pairs $(X,B)$ such that $\mld(X,B)<a$ and there exists only $1$ exceptional divisor with log discrepancy $\le a$ with respect to $(X,B)$, was constructed. When $a=1$, these are exactly enc pairs. However, \cite{Liu18} only deals with the case when $a<1$ and does not deal with the case when $a=1$. \cite[Definition 2.1]{Jia21} first formally introduced enc varieties $X$, naming them ``extremely noncanonical". In dimension $3$ and when $\mld(X)\to 1$, \cite{Jia21} systematically studied the singularities of these varieties, which played a crucial role in his proof of the $1$-gap conjecture for threefolds. \cite{HLL22} introduced enc pairs $(X,B)$ to prove the $1$-gap conjecture for threefold pairs. 
\end{rem}

\begin{rem}[Enc pairs and exceptional Fano pairs]\label{rem: enc and exceptional pair}
We explain why we use the notation ``exceptionally non-canonical" instead of ``extremely noncanonical" as in \cite{Jia21}. The key reason is that, as suggested by Shokurov, we expect exceptionally non-canonical singularities in dimension $d$ to have connections with the global lc thresholds in dimension $d-1$, while the latter is known to have connections with the mlds of exceptional pairs in dimension $d-1$ \cite[Theorem 1.2]{Liu23} (cf. \cite{HLS24,Sho20}). Shokurov suggested us that the role of enc singularities in the study of klt singularities may be as important as the role of exceptional pairs in the study of Fano varieties (cf. \cite{Bir19}).

Theorem \ref{thm: ecn 1 gap implies klt 1 gap} could provide some evidence of this for us: when $d=3$ and $\Ii=\{0\}$, the $1$-gap of mld is equal to $\frac{1}{13}$ (Theorem \ref{thm: 12/13}) and is reached at an enc cyclic quotient singularity $\frac{1}{13}(3,4,5)$. If we let $f: Y\rightarrow X$ be the divisorial contraction which extracts the unique prime divisor $E$ over $X\ni x$ such that $a(E,X,0)=\mld(X\ni x)$, then $E$ is normal and $$\left(\mathbb P(3,4,5),\frac{12}{13}(x_1^3x_2+x_2^2x_3+x_3^2x_1=0)\right)\cong (E,B_E),$$
where $K_E+B_E\sim_{\mathbb Q}f^*K_X|_E=(K_Y+\frac{1}{13}E)|_E$. On the other hand, $\frac{12}{13}$ is also expected to be \footnote{It is proven in \cite[Theorem 1.1]{LS23} after the first version of this paper appeared.} the largest surface global lc threshold (\cite[Remark 2.5]{Liu23},~\cite[Notation 4.1]{AL19}) and can be reached by the same pair $(\mathbb P(3,4,5),\frac{12}{13}(x_1^3x_2+x_2^2x_3+x_3^2x_1=0))$ \cite[40]{Kol13}.
\end{rem}

\begin{rem}[Enc pairs, Calabi-Yau varieties, and mirror symmetry]\label{rem: enc and cy varieties}
Enc pairs also have a deep relationship with Calabi-Yau varieties in different ways. 

First, by Theorem \ref{thm: ecn 1 gap implies klt 1 gap}, the $1$-gap conjecture for mlds of enc pairs implies the $1$-gap conjecture of mlds, while the latter will imply the birational boundedness of rationally connected Calabi-Yau varieties by applying similar arguments as in \cite[Proof of Theorem 1.2]{HJ24}.

Second, as mentioned in Remark \ref{rem: enc and exceptional pair}, the mlds of enc pairs have connections with the global lc thresholds, while the latter is related to the minimal possible mld of klt Calabi-Yau varieties. Indeed, the $1$-gap of the mlds of enc pairs is always smaller than or equal to the minimal possible mld of klt Calabi-Yau varieties of smaller dimensions, and they are expected to be the same (cf. \cite[Proposition 6.1]{ETW22}). Finally, the second author was informed by Chengxi Wang that the klt Calabi-Yau variety with minimal possible mld should be associated with a klt Calabi-Yau variety with maximal possible index by mirror symmetry (cf. \cite[Proposition 6.1]{ETW22}).
\end{rem}


\begin{thebibliography}{99}

\bibitem[Ale93]{Ale93} V.~Alexeev, \textit{Two two--dimensional terminations}, Duke Math. J. \textbf{69} (1993), no. 3, 527--545.


\bibitem[AHK07]{AHK07} V.~Alexeev, C. D.~Hacon, and Y.~Kawamata, \textit{Termination of (many) $4$-dimensional log flips}, Invent. Math. \textbf{168} (2007), no. 2, 433--448.

\bibitem[AL19]{AL19} V. Alexeev and W. Liu, \textit{Open surfaces of small volume}, Algebraic Geom. \textbf{6} (2019), no. 3, 312--327.

\bibitem[Amb99]{Amb99} F.~Ambro, \textit{On minimal log discrepancies}, Math. Res. Lett. \textbf{6} (1999), 573–580.


\bibitem[Amb06]{Amb06} F.~Ambro, \textit{The set of toric minimal log discrepancies}. Cent. Eur. J. Math. \textbf{4} (2006), no. 3, 358--370.

	
\bibitem[Bir19]{Bir19} C. Birkar, \textit{Anti-pluricanonical systems on Fano varieties}. Ann. of Math. (2), \textbf{190} (2019), 345--463.

\bibitem[BCHM10]{BCHM10}
C. Birkar, P. Cascini, C. D. Hacon and J. M\textsuperscript{c}Kernan, \textit{Existence of minimal models for varieties of log general type}, J. Amer. Math. Soc. \textbf{23} (2010), no. 2, 405--468.



\bibitem[BZ16]{BZ16} C. Birkar and D.-Q. Zhang, \textit{Effectivity of Iitaka fibrations and pluricanonical systems of polarized pairs}, Pub. Math. IHES., \textbf{123} (2016), 283--331.


\bibitem[CH21]{CH21} G.~Chen and J.~Han,
\textit{Boundedness of $(\epsilon, n)$-complements for surfaces}, arXiv:2002.02246. Short version published in Adv. Math. \textbf{383} (2021), 107703, 40pp.




\bibitem[CT23]{CT23} G. Chen and N.~Tsakanikas, \textit{On the termination of flips for log canonical generalized pairs}, Acta. Math. Sin. English Ser. \textbf{39} (2023), 967--994. 

\bibitem[Che22]{Che22} J.-J. Chen, \textit{Accumulation points on 3-fold canonical thresholds}, arXiv:2202.06230. To appear in J. Math. Soc. Jpn. \url{https://www.mathsoc.jp/publication/JMSJ/pdf/JMSJ9150.pdf}

\bibitem[CDCHJS21]{CDCHJS21} W.~Chen, G.~Di Cerbo, J.~Han, C.~Jiang, and R.~Svaldi, {\it Birational boundedness of rationally connected Calabi--Yau $3$-folds}, Adv. Math. \textbf{378} (2021), Paper No. 107541, 32 pp.

\bibitem[CGN24]{CGN24} W. Chen, Y. Gongyo, and Y. Nakamura, \textit{On generalized minimal log discrepancy}, J. Math. Soc. Japan \textbf{76} (2024), 393--449.


\bibitem[ETW22]{ETW22} L. Esser, B. Totaro, and C. Wang, \textit{Calabi-Yau varieties of large index}, arXiv:2209.04597.

\bibitem[EMY03]{EMY03} L. Ein, M. Musta\c{t}\v{a}, and T. Yasuda, \textit{Jet  schemes,  log  discrepancies  and  inversion  of adjunction}, Invent. Math. \textbf{153} (2003), no. 3, 519--535.


\bibitem[Fuj04]{Fuj04} O. Fujino, \textit{Termination of $4$-fold canonical flips}, Publ. Res. Inst. Math. Sci, \textbf{40} (2004), no. 1, 231--237.

\bibitem[Fuj05]{Fuj05} O. Fujino, \textit{Addendum to ``Termination of $4$-fold canonical flips”}, Publ. Res. Inst. Math. Sci. \textbf{41} (2005), no. 1, 252--257.

\bibitem[Fuj07]{Fuj07} O. Fujino, \textit{Special termination and reduction to pl flips}, from ``Flips for 3–folds and 4–folds” (A. Corti, editor), Oxford Lect. Ser. Math. Appl. \textbf{35}, Oxford Univ. Press (2007), 63--75.


\bibitem[Fuj17]{Fuj17} O. Fujino, \textit{Foundations of the minimal model program}, MSJ Memoirs, \textbf{35}. Mathematical Society of Japan, Tokyo (2017).



\bibitem[Ful98]{Ful98} W. Fulton, \textit{Intersection theory}. Second edition. Ergebnisse der Mathematik und ihrer Grenzgebiete. 3. Folge. A Series of Modern Surveys in Mathematics [Results in Mathematics and Related Areas. 3rd Series. A Series of Modern Surveys in Mathematics], 2. Springer-Verlag, Berlin, 1998.


\bibitem[HL23a]{HL23a} C. D. Hacon and J. Liu, \textit{Existence of flips for generalized lc pairs}, Camb. J. Math. \textbf{11} (2023), no. 4, 795--828.  

\bibitem[HMX14]{HMX14} C. D. Hacon, J.~M\textsuperscript{c}Kernan, and C.~Xu, \textit{ACC for log canonical thresholds}, Ann. of Math. \textbf{180} (2014), no. 2, 523--571.

\bibitem[HJ24]{HJ24} J.~Han and C.~Jiang, \textit{Birational boundedness of rationally connected log Calabi-Yau pairs with fixed index}, Algebraic Geometry and Physics, \textbf{1} (2024), no. 1, 59--79.


\bibitem[HL22]{HL22} J.~Han and Z.~Li, \textit{Weak Zariski decompositions and log terminal models for generalized polarized pairs}, Math. Z. \textbf{302} (2022), 707--741.


\bibitem[HLQ21]{HLQ21} J.~Han, Z.~Li, and L.~Qi, \textit{ACC for log canonical threshold polytopes}, Amer. J. Math. \textbf{143} (2021), no. 3, 681--714.


\bibitem[HL25]{HL25} J. Han and J. Liu, \textit{Classification of threefold enc cDV quotient singularities}, arXiv:2501.01024.


\bibitem[HLL22]{HLL22} J.~Han, J.~Liu, and Y.~Luo, \textit{ACC for minimal log discrepancies of terminal threefolds}, arXiv:2202.05287.


\bibitem[HLS24]{HLS24} J.~Han, J.~Liu, and V. V. Shokurov, \textit{ACC for minimal log discrepancies of exceptional singularities}, Peking Math. J. (2024) \url{https://link.springer.com/article/10.1007/s42543-024-00091-x}.


\bibitem[HLQ23]{HLQ23} J.~Han, Y.~Liu, and L.~Qi, \textit{ACC for local volumes and boundedness of singularities}, J. Algebraic Geom. \textbf{32} (2023), 519--583. 

\bibitem[HL23b]{HL23b} J.~Han and Y.~Luo,
\textit{On boundedness of divisors computing minimal log discrepancies for surfaces}, arXiv:2002.02246. Short version published on J. Inst. Math. Jussieu (2023), \textbf{22}(6), 2907--2930. 

\bibitem[HL24]{HL24} J.~Han and Y.~Luo,
\textit{A Simple Proof of ACC for Minimal Log Discrepancies for Surfaces}, Acta Math. Sin. (Engl. Ser.) \textbf{40} (2024) 425--434. 


\bibitem[Hay99]{Hay99}T.~Hayakawa, \textit{Blowing ups of 3-dimensional terminal singularities}, Publ. Res. Inst. Math. Sci. \textbf{35} (1999), no. 3, 515--570.



\bibitem[Jia21]{Jia21} C.~Jiang, \textit{A gap theorem for minimal log discrepancies of non-canonical singularities in dimension three}, J. Algebraic Geom. \textbf{30} (2021), 759--800.



\bibitem[Kaw05]{Kaw05} M.~Kawakita, \textit{Three-fold divisorial contractions to singularities of higher indices}, Duke Math. J. \textbf{130} (2005), no. 1, 57--126.


\bibitem[Kaw15a]{Kaw15a} M.~Kawakita, \textit{The index of a threefold canonical singularity}, Amer. J. Math. \textbf{137} (2015), no. 1, 271--280.

\bibitem[Kaw15b]{Kaw15b} M.~Kawakita, \textit{A connectedness theorem over the spectrum of a formal power series ring}, Internat. J. Math. \textbf{26} (2015), no. 11, 1550088, 27 pp.


\bibitem[Kaw88]{Kaw88} Y.~Kawamata, \textit{Crepant blowing-up of $3$-dimensional canonical singularities and its application to degenerations of surfaces}, Ann. of Math. (2), \textbf{127} (1988), no. 1, 93--163.

\bibitem[Kaw92a]{Kaw92a} Y.~Kawamata, \textit{Termination  of  log  flips  for  algebraic $3$-folds}, Internat. J. Math. \textbf{3} (1992), no. 5, 653--659.

\bibitem[Kaw92b]{Kaw92b} Y.~Kawamata, \textit{The minimal discrepancy of a 3-fold terminal singularity} (Russian), Appendix to Shokurov, V. V., 3-fold log flips, Izv. Ross. Akad. Nauk Ser. Mat. \textbf{56} (1992), no. 1, 105--203. 


\bibitem[KMM87]{KMM87} Y.~Kawamata, K.~Matsuda, and K.~Matsuki, \textit{Introduction to the minimal model problem}, Algebraic geometry, Sendai, 1985, 283--360, Adv. Stud. Pure Math., \textbf{10}, North-Holland, Amsterdam, 1987.


\bibitem[Kol$^+$92]{Kol+92} J.~Koll\'{a}r \'{e}t al., \textit{Flip and abundance for algebraic threefolds}, Ast\'{e}risque no. \textbf{211}, 1992.

\bibitem[Kol96]{Kol96} J. Koll\'ar, \textit{Rational curves on algebraic varieties}, Ergebnisse der Mathematik und ihrer Grenzgebiete. 3. Folge. A Series of Modern Surveys in Mathematics [Results in Mathematics and Related Areas. 3rd Series. A Series of Modern Surveys in Mathematics], 32. Springer-Verlag, Berlin, 1996. 

\bibitem[Kol13]{Kol13} J.~Koll\'ar, \textit{Moduli of varieties of general type}, Handbook of moduli \textbf{2}, Adv. Lect. Math. (ALM), \textbf{25} (2013), Int. Press, Somerville, MA, 131--157.


\bibitem[KM98]{KM98} J.~Koll\'{a}r and S.~Mori, \textit{Birational geometry of algebraic varieties}, Cambridge Tracts in Math. \textbf{134} (1998), Cambridge Univ. Press.

\bibitem[KSB88]{KSB88} J.~Koll\'ar and N.I.~Shepherd-Barron, \textit{Threefolds and deformations of surface singularities}, Invent. Math. \textbf{91} (1988), no. 2, 299--338.

\bibitem[LLX20]{LLX20} C. Li, Y. Liu, and C. Xu, \textit{A guided tour to normalized volume}, Geometric analysis-- in honor of Gang Tian's 60th birthday, Progr. Math. \textbf{333} (2020), Birkh\"auser/Springer, Cham, 167--219, 

\bibitem[Liu18]{Liu18} J.~Liu, \textit{Toward the equivalence of the ACC for a-log canonical thresholds and the ACC for minimal log discrepancies}, arXiv:1809.04839.

\bibitem[LL22]{LL22} J.~Liu and Y.~Luo, \textit{Second largest accumulation point of minimal log discrepancies of threefold}, arXiv:2207.04610.

\bibitem[LS23]{LS23} J. Liu and V. V. Shokurov, \textit{Optimal bounds on surfaces}, arXiv:2305.19248. To appear in Algebraic Geometry and Physics.

\bibitem[LX21]{LX21} J.~Liu and L.~Xiao, \textit{An optimal gap of minimal log discrepancies of threefold non-canonical singularities}, J. Pure Appl. Algebra \textbf{225} (2021), no. 9, 106674, 23 pp.


\bibitem[Liu23]{Liu23} J. Liu, \textit{Remark on complements on surfaces}, Forum Math. Sigma \textbf{11} (2023), E42.



\bibitem[Mar96]{Mar96} D.~Markushevich, \textit{Minimal discrepancy for a terminal cDV singularity is 1}, J. Math. Sci. Univ. Tokyo \textbf{3} (1996), no. 2, 445--456.

\bibitem[Mor85]{Mor85} S.~Mori, \textit{On 3-dimensional terminal singularities}, Nagoya Math J. \textbf{98} (1985), 43--66.



\bibitem[Nak16]{Nak16} Y.~Nakamura, \textit{On minimal log discrepancies on varieties with fixed Gorenstein index}, Michigan Math. J. \textbf{65} (2016), no. 1, 165--187.


\bibitem[NS21]{NS21} Y.~Nakamura and K.~Shibata, \textit{Inversion of adjunction for quotient singularities II: Non-linear actions}, arXiv:2112.09502. To appear in Algebr. Geom.


\bibitem[NS22]{NS22} Y.~Nakamura and K.~Shibata, \textit{Inversion of adjunction for quotient singularities}, Algebr. Geom. \textbf{9} (2) (2022) 214--251.


\bibitem[Rei87]{Rei87} M.~Reid, \textit{Young person’s guide to canonical singularities}, Algebraic geometry, Bowdoin 1985, Proc. Symp. Pure Math. \textbf{46} (1987), Part 1, 345--414.

\bibitem[Sho85]{Sho85} V. V.~Shokurov, \textit{A nonvanishing theorem}, Izv. Akad. Nauk SSSR. Set. Mat., \textbf{49} (1985), 635--651.

\bibitem[Sho88]{Sho88} V. V.~Shokurov, {\it Problems about {F}ano varieties}. {Birational Geometry of Algebraic Varieties, Open Problems. The XXIIIrd International Symposium, Division of Mathematics, The Taniguchi Foundation}, 30--32, August 22--August 27, 1988.

\bibitem[Sho94a]{Sho94a} V. V.~Shokurov, \textit{Semi-stable 3-fold flips}, Izv. Ross. Akad. Nauk Ser. Mat. \textbf{57} (1993), no. 2, 165--222; reprinted in
Russian Acad. Sci. Izv. Math. \textbf{42} (1994), no. 2, 371--425.

\bibitem[Sho94b]{Sho94b} V. V.~Shokurov, \textit{A.c.c. in codimension 2} (1994), preprint.

\bibitem[Sho96]{Sho96} V. V.~Shokurov, \textit{3-fold log models}, J. Math. Sci. \textbf{81} (1996), no. 3, 2667--2699.

\bibitem[Sho04]{Sho04} V. V.~Shokurov, \textit{Letters of a bi-rationalist, V. Minimal log discrepancies and termination of log flips} (Russian), Tr. Mat. Inst. Steklova \textbf{246} (2004), Algebr. Geom. Metody, Svyazi i Prilozh., 328--351.

\bibitem[Sho20]{Sho20} V. V.~Shokurov, \textit{Existence and boundedness of $n$-complements}, arXiv:2012.06495.

\bibitem[Zhu24]{Zhu24} Z.~Zhuang, \textit{On boundedness of singularities and minimal log discrepancies of Koll\'ar components}, J. Algebraic Geom. \textbf{33} (2024), 521--565.

\end{thebibliography}
\end{document}